\documentclass[reqno]{amsart}

\usepackage{amscd, color, mathrsfs,enumerate}

\usepackage{comment}
\usepackage{amsfonts,latexsym,amstext}
\usepackage{amsmath, amssymb,color}
\usepackage{amsthm}

\newtheorem{theorem}{Theorem}[section]
\newtheorem{lemma}[theorem]{Lemma}
\newtheorem{proposition}[theorem]{Proposition}
\newtheorem{definition}[theorem]{Definition}
\newtheorem{example}[theorem]{Example}
\newtheorem{hypothesis}[theorem]{Hypothesis}

\newtheorem{remark}[theorem]{Remark}
\newtheorem{corollary}[theorem]{Corollary}

\numberwithin{equation}{section}

\def\sqr#1#2{{\vcenter{\vbox{\hrule height .#2pt \hbox{\vrule
 width .#2pt height#1pt \kern#1pt \vrule
width .#2pt} \hrule height .#2pt}}}}

\def\ds{\begin{displaystyle}}
\def\eds{\end{displaystyle}}

\def\<{\langle }
\def\>{\rangle }

\def\R{\mathbb R}
\def\N{\mathbb N}

\allowdisplaybreaks

\begin{document}

\title[Regularity results for non-linear Young equations]{Regularity results for non-linear Young equations and applications}

\author[D. Addona]{Davide Addona}
\address{D.A.: Dipartimento di Scienze Matematiche, Fisiche e Informatiche, Plesso di Matema\-ti\-ca, Universit\`a degli Studi di Parma, Viale Parco Area delle Scienze 53/A, I-43124 Parma, Italy}
\author[L. Lorenzi]{Luca Lorenzi}
\address{L.L.: Dipartimento di Scienze Matematiche, Fisiche e Informatiche, Plesso di Matema\-ti\-ca, Universit\`a degli Studi di Parma, Viale Parco Area delle Scienze 53/A, I-43124 Parma, Italy}
\author[G. Tessitore]{Gianmario Tessitore}
\address{G.T.: Dipartimento di Matematica e Applicazioni, Universit\`a degli Studi di Milano-Bicocca, Milano, Via R. Cozzi 53, I-20126 Milano Italy}
\email{davide.addona@unipr.it; https://orcid.org/0000-0002-6372-0334}
\email{luca.lorenzi@unipr.it; https://orcid.org/0000-0001-6276-5779}
\email{gianmario.tessitore@unimib.it; https://orcid.org/0000-0001-9893-3703}

\begin{abstract}
In this paper we provide sufficient conditions which ensure that the non-linear equation $dy(t)=Ay(t)dt+\sigma(y(t))dx(t)$, $t\in(0,T]$, with $y(0)=\psi$ and $A$ being an unbounded operator, admits a unique mild solution which is classical, i.e., $y(t)\in D(A)$ for any $t\in (0,T]$, and we compute the blow-up rate of the norm of $y(t)$ as $t\rightarrow 0^+$. We stress that the regularity of $y$ is independent on the smoothness of the initial datum $\psi$, which in general does not belong to $D(A)$. As a consequence we get an integral representation of the mild solution $y$ which allows us to prove a chain rule formula for smooth functions of $y$ and necessary conditions for the invariance of hyperplanes with respect to the non-linear evolution equation.
\end{abstract}

\maketitle

\section{Introduction}

The Young integral has been introduced in \cite{You36}, where the author defines extension of the Riemann-Stieltjes  integral $\int fdg$ when neither $f$ nor $g$ have finite total variations. In particular in \cite{You36} it is shown that,  if $f$ and $g$ are continuous functions such that $f$ has finite $p$-variation and $g$ has finite $q$-variation, with $p,q>0$ and $p^{-1}+q^{-1}>1$, then the Stieltjes integral $\int fdg$ is well-defined as a limit of  Riemann sums.
This was the starting point of the crucial extension to rough paths integration. Indeed, in \cite{Lyo98} the author proves that it is possible to define the integral $\int fdx$ also in the case when $f$ has finite $p$-variation and $x$ has finite $q$-variation with $p,q>0$ and $p^{-1}+q^{-1}<1$. In this case, additional information on the function $x$ is needed, which would play the role of iterated integrals for regular paths.

An alternative formulation of the integration over rough paths is provided in \cite{G04}, where the author considers H\"older-like (semi)norms instead of $p$-variation norms. Namely, if $ f$ is $\alpha$-H\"older continuous and $g$ is $\eta$-H\"older continuous with $\alpha+\beta>1$ then  the Young integral is well defined  as the unique solution to an algebraic problem. Recently, a more general theory of rough integration, when   $\alpha+\beta\leq 1$, has been introduced in \cite{FrHa}.

Here, we consider only Young integrals and focus on the spatial regularity of solutions to infinite dimensional evolution equations leaving aside the enormous amount of results connected to the rough paths case culminating in the   breakthrough on singular SPDEs (see, e.g., \cite{gubinelli-panorama}).
Namely, we consider the non-linear evolution equation
\begin{align}
\label{omo_equation-intro}
\left\{
\begin{array}{ll}
dy(t)=Ay(t)dt+\sigma(y(t))dx(t), &t\in(0,T],\\[1mm]
y(0)=\psi,
\end{array}
\right.
\end{align}
where $A$ is the infinitesimal generator of a semigroup with suitable regularizing properties and $x$ is a $\eta$-H\"older continuous function with $\eta>1/2$.
Ordinary differential equations (in finite dimensional spaces) driven by
an irregular path of H\"older regularity greater than 1/2 have been understood in full details since  \cite{Zah98} (see also \cite{Lej03}). On the other hand, the infinite dimensional case was treated in \cite{GLT06} and then developed in \cite{GT10} and \cite{DGT12}, see also \cite{MaNu03} for earlier results in the context of
stochastic partial differential equations driven by an infinite dimensional
fractional Brownian motion of Hurst parameter $H>1/2$.

In \cite{DGT12}, problem \eqref{omo_equation-intro} is formulated in a mild form
\begin{eqnarray*}
y(t)=S(t)y(0)+\int_0^t S(t-r)(\sigma(y(r)))dx(r), \qquad\;\, t\in[0,T],
\end{eqnarray*}
where $(S(t))_{t\ge 0}$ is the analytic semigroup generated by the sectorial operator $A$, and the authors exploit the regularizing properties  of $S$ to show that, if the initial datum $\psi$ is smooth enough (i.e., if it belongs to a suitable domain of the fractional powers $(-A)^{\alpha}$), then equation \eqref{omo_equation-intro} admits a unique mild solution with the same spatial regularity as the initial datum. The key technical point in \cite{DGT12} is to prove that the convolution
\begin{equation}\label{convolution-intro}
 \int_0^tS(t-s)f(s)dx(s)
\end{equation}
is well defined if $f$ takes values in $D((-A)^{\alpha})$ and belongs to a H\"older-type function space. To be more precise, the authors require that $f:[0,T]\rightarrow D((-A)^\alpha)$ satisfies the condition
\begin{align*}
\sup_{s<t, s,t\in[0,T]}\frac{|f(t)-f(s)-(S(t-s)-I)f(s)|_{D((-A)^\alpha)}}{(t-s)^{\beta}}<+\infty.
\end{align*}
This is one of the main difference with respect to the finite dimensional case, where the condition on the function $f$ reads in terms of classical H\"older norms.
Once that convolution \eqref{convolution-intro} is well-defined,  the smoothness of the initial datum $\psi$ and suitable estimates on \eqref{convolution-intro} allow the authors to solve the mild reformulation of equation \eqref{omo_equation-intro}  by a fixed point argument in the same H\"older-type function space introduced above.

Our point in the present paper is that if one looks a bit more closely to the trade-off between  H\"olderianity in time and regularity in space of the convolution \eqref{convolution-intro} one discovers that an extra regularity in space can be extracted by estimates, see Lemma \ref{lem:div_int}. This allows us to show that the mild solution to equation \eqref{omo_equation-intro}, which in our situation is driven by a finite dimensional noise, is more regular than the initial datum (that nevertheless has to enjoy the same regularity assumptions as in \cite{DGT12}). Namely, it is indeed a classical solution, i.e., $y(t)\in D(A)$ for any $t\in(0,T]$ (see Theorem \ref{es_mild_sol_gub}), and, as a byproduct, we prove that $y$ is also an integral solution to \eqref{omo_equation-intro}, i.e., it satisfies the equation
\begin{align*}
    y(t)=\psi+\int_0^t Ay(s)ds+\int_0^t \sigma(y(s))dx(s),\qquad\;\,t\in [0,T].
\end{align*}

It is worth mentioning that, when $A$ is an unbounded operator,  mild formulation of equation \eqref{omo_equation-intro} is the most suitable to prove existence and uniqueness of a solution since it allows to apply a fixed point argument in spaces of functions with a low degree of smoothness. On the other hand, this formulation is too weak in several  applications (for instance  Section 6 here), where an integral formulation of the solution helps a lot. Moreover, since our solution is classical, we can also obtain a chain rule; in other words we show that we can differentiate with respect to time regular enough functions of the solution to equation \eqref{omo_equation-intro}. Finally, the chain rule turns out to be the right tool to obtain necessary conditions for invariance of half spaces with respect to \eqref{omo_equation-intro} (see \cite{CouMar16} for general invariance results in the case of a finite dimensional equation driven by a rough path and \cite{CaDa} for the state of the art in the case of classical evolution equations in infinite dimensional spaces).

Summarizing, this paper can been described as a first step towards a systematic study, by the classical tools of semigroup theory, of smoothing properties of the mild solution to  \eqref{omo_equation-intro} (as they were introduced in \cite{DGT12}). We plan to go further in the analysis, first weakening the smoothness assumptions on $\psi$ and, then, developing results analogous to those in this paper for equations driven by more irregular noises as in the case of rough paths.

The paper is structured as follows.
In Section \ref{sect-2}, we introduce the function spaces that we use in this paper and we recall some results taken from \cite{G04,GT10}, slightly generalizing some of those results.
In Section \ref{sect-3}, we prove the existence and uniqueness of a mild solution to the non-linear Young equation \eqref{omo_equation-intro} when $\psi\in X_{\alpha}$ and $x$ is $\eta$-H\"older continuous for some $\eta\in(1/2,1)$ and $\alpha+\eta>1$. We show that this solution is classical and estimate the blow-up rate of its $X_{1+\mu}$-norm as $t$ tends to $0^+$, when $\mu\in [0,\eta+\alpha-1)$. Based on this result, in Section \ref{sect-4} we
prove that the mild solution to \eqref{omo_equation-intro} can be written in an integral form, which is used in Section \ref{sect-5} to prove the chain rule.  As a byproduct of the previous results, in Section \ref{sect-6} we prove a necessary condition for the invariance of the halfplane. Finally, in Section \ref{sect-examples} we provide an example to illustrate our results.

\medskip

\paragraph{\bf Notation.} 
We denote by $[a,b]^2_<$ the set $\{(s,t)\in\R^2:a\leq s<t\leq b\}$.
Further, for every $h\in X_\alpha$ we set $|h|_\alpha:=|h|_{X_\alpha}$ and we denote by $\mathscr L(X_\alpha,X_\gamma)$ the space of linear bounded operators from $X_\alpha$ into $X_\gamma$, for each $\alpha,\gamma\in\R$. For every operator $Q\in \mathscr L(X_\gamma;X_\alpha)$, we denote by $|Q|_{\gamma,\alpha}$ its operator norm.
For every $A\subset\R$,  $C(A;X)$ denotes the usual space of continuous functions from $A$ into $X$ endowed with the sup-norm. The subscript ``$b$'' stands for bounded. When $X=X_{\alpha}$ we denote the sup-norm by $\|\cdot\|_{\alpha,A}$.
If $f\in C^{\alpha}([a,b])$ for some $\alpha\in (0,1)$, then we denote by $[f]_{\alpha}$ the usual $\alpha$-H\"older seminorm of $f$.

\section{The abstract Young equation}
\label{sect-2}

\subsection{Function spaces and preliminary results}
Throughout the paper, $X$ denotes a Banach space and $A:D(A)\subseteq X\rightarrow X$ is a linear operator which generates a semigroup $(S(t))_{t\geq 0}$. We further assume the following set of assumptions.

\begin{hypothesis}
\label{hyp-main}
\begin{enumerate}[\rm (i)]
\item
For every $\alpha\in [0,2)$, there exists a space $X_\alpha$ $($with the convention that $X_0=X$ and $X_1=D(A))$ such that if $\beta\leq \alpha$ then $X_\alpha$ is continuously embedded into $X_\beta$. We denote by $K_{\alpha,\beta}$ a positive constant such that $|x|_{\beta}\le K_{\alpha,\beta}|x|_{\alpha}$ for every $x\in X_{\alpha}$;
\item
for every $\zeta,\alpha,\gamma\in [0,2)$, $\zeta\leq\alpha$, and $\mu,\nu\in(0,1]$ with $\mu>\nu$ there exist positive constants $M_{\zeta,\alpha,T}$, and $C_{\mu,\nu,T}$, which depend on $T$,such that
\begin{align}
\left\{
\begin{array}{ll}
(a)\ \|S(t)\|_{\mathscr{L}(X_\zeta, X_\alpha)}\leq M_{\zeta,\alpha,T} t^{-\alpha+\zeta},\\[1mm]
(b) \ \|S(t)-I\|_{\mathscr{L}(X_\mu,X_\nu)}\leq C_{\mu,\nu,T} t^{\mu-\nu},
\end{array}
\right.
\label{stime_smgr}
\end{align}
for every $t\in (0,T]$.
\end{enumerate}
\end{hypothesis}

\begin{example}
{\rm If $A$ is a sectorial operator on $X$, with resolvent set which contains a sector centered at the origin, then Hypothesis \ref{hyp-main} are satisfied if we set $X_\alpha:=D_A(\alpha,\infty)$ for every $\alpha\in (0,2)$}.
\end{example}

We now introduce some operators which will be used extensively in this paper.
\begin{definition}
Let $a$ and $b$ be two real numbers with $a<b$. Then,
the operators $\delta_1,\hat\delta_1:C([a,b];X)\rightarrow C([a,b]^2_<;X)$ are defined as follows:
\begin{align*}
& (\delta_1f)(s,t)= f(t)-f(s),\\[1mm]
&  (\hat\delta_1 f)(s,t)= (\delta_1f)(s,t)-{\mathfrak a}(s,t)f(s),
\end{align*}
for every $(s,t)\in [a,b]^2_<$ and $f\in C([a,b];X)$, where $\mathfrak{a}(s,t)=S(t-s)-I$.
\end{definition}

\begin{remark}
{\rm We stress that the continuity of the function $\mathfrak{a}$ in $[a,b]^2_{<}$ is implied by the strong continuity of the semigroup $(S(t))_{t\geq0}$ in
$(0,+\infty)$. No continuity assumptions at $t=0$ is required.}
\end{remark}

\subsection{Function spaces}

\begin{definition}
For every $a,b\in\R$, with $a<b$ and $\alpha,\beta\in [0,2)$, we denote by:
\begin{enumerate}[\rm (i)]
\item
$C_{\beta}([a,b]_<^2;X_\alpha)$ the subspace of $C([a,b]^2_<;X_{\alpha})$ consisting of functions $f$ such that
\begin{align*}
\|f\|_{\beta|\alpha,[a,b]}:=\sup_{(s,t)\in [a,b]^2_{<}}\frac{|f(s,t)|_{\alpha}}{|t-s|^\beta}<+\infty;
\end{align*}
when there is no risk of confusion, we simply write $\|f\|_{\beta|\alpha}$ instead of $\|f\|_{\beta|\alpha,[a,b]}$. Notice that, if $g\in C^{\beta}([a,b];X_{\alpha})$, then $\|\delta_1g\|_{\beta|\alpha}$ is the standard $\beta$-H\"older seminorm of $g$;
\item
$\hat C_{\beta}([a,b];X_\alpha)$ the subset of $C([a,b];X_\alpha)$ consisting of functions $f$ such that $\hat\delta_1 f\in C_{\beta}([a,b]_<^2;X_\alpha)$ endowed with the norm
\begin{align*}
\|f\|_{\beta,\alpha,[a,b]}:=\|f\|_{\alpha,[a,b]}+\|\hat\delta_1 f\|_{\beta|\alpha,[a,b]}.
\end{align*}
\end{enumerate}
\end{definition}

\begin{remark}
{\rm For every $a,b\geq0$ with $a<b$, and $\alpha, \beta, k\in [0,2)$ the following properties hold true.
\begin{enumerate}[\rm (i)]
\item
If $f\in C([a,b];X_{\alpha})\cap \hat C_k([a,b];X_\beta)$ then $f\in C^\rho([a,b];X_\gamma)$ for every $\gamma\in [0,\beta]$, such that $\gamma<\alpha$, and $\rho:=\min\{k,\alpha-\gamma\}$. Indeed, for every $(s,t)\in[a,b]_<^2$ we can estimate
\begin{align*}
|f(t)-f(s)|_\gamma
\leq |(\hat\delta_1f)(s,t)|_\gamma+|\mathfrak a(s,t)f(s)|_\gamma.
\end{align*}
Estimating separately the two terms we get
\begin{align*}
&|(\hat\delta_1f)(s,t)|_\gamma\leq K_{\beta,\gamma}\|f\|_{k|\beta,[a,b]}|t-s|^k,\\
&|\mathfrak a(s,t)f(s)|_\gamma\leq C_{\alpha,\gamma,b}\|f\|_{\alpha,[a,b]}|t-s|^{\alpha-\gamma}
\end{align*}
for every $a\le s<t\le b$, which yields the assertion. In particular, $\hat C_{\alpha}([a,b];X_{\alpha})$ is continuously embedded into $C^{\alpha-\gamma}([a,b];X_{\gamma})$ if $\alpha\in (0,1)$ and $\gamma\in [0,\alpha]$, it is contained in the space of Lipschitz continuous functions over $[a,b]$ with values in $X$, if $\alpha=1$, and it consists of constant functions if $\alpha>1$.
\item
For every $f:[a,b]_<^2\rightarrow X$ and $\alpha,\beta,\gamma\geq0 $, such that $\beta>\gamma$, it holds that
\begin{align*}
\|f\|_{\gamma|\alpha,[a,b]}\leq |b-a|^{\beta-\gamma}\|f\|_{\beta|\alpha,[a,b]}.
\end{align*}
Indeed, if $\|f\|_{\gamma|\alpha,[a,b]}=+\infty$ or $\|f\|_{\beta|\alpha,[a,b]}=+\infty$ then the statement is trivial. Let us assume that both $\|f\|_{\gamma|\alpha,[a,b]}$ and $\|f\|_{\beta|\alpha,[a,b]}$ are finite. Then,
\begin{align*}
\|f\|_{\gamma|\alpha,[a,b]}
= & \sup_{(s,t)\in[a,b]_<^2}\frac{|f(s,t)|_\alpha}{|t-s|^{\gamma}}
= \sup_{(s,t)\in[a,b]_<^2}\frac{|f(s,t)|_\alpha}{|t-s|^{\beta}}|t-s|^{\beta-\gamma} \\
\leq & |b-a|^{\beta-\gamma} \|f\|_{\beta|\alpha,[a,b]}.
\end{align*}
\end{enumerate}}
\end{remark}

We recall some relevant results from \cite{G04} and \cite{GT10}. In particular, we recall the definition of the Young integrals
\begin{align*}
\int_s^tf(r)dx(r), \qquad\;\, \int_s^tS(t-r)f(r)dx(r), \quad s,t\in[a,b],
\end{align*}
where $f:[a,b]\rightarrow X$ and $x:[a,b]\rightarrow \R$ satisfy suitable assumptions.
In particular, we assume the following condition on $x$.

\begin{hypothesis}
\label{ip:young_path}
$x\in C^\eta([a,b])$ for some $\eta\in (1/2,1)$.
\end{hypothesis}

\begin{theorem}[Section $3$ in \cite{G04} and Section $2$ in \cite{GT10}]
\label{thm:young_int}
Fix $f\in C^\alpha([a,b];X)$, where $\alpha\in (1-\eta,1)$. Then, for each $(s,t)\in [a,b]_<^2$ the Riemann series
\begin{align}
\sum_{i=0}^{n-1}f(t_i)(x(t_{i+1})-x(t_i)),
\label{somme_di_riemann-0}
\end{align}
where $\Pi(s,t):=\{t_0=s<t_1<\ldots<t_n=t\}$ is a partition of $[s,t]$ and $|\Pi(s,t)|:=\max\{t_{i+1}-t_i: i:=0,\ldots,n-1\}$, converges in $X$ as $|\Pi(s,t)|$ tends to $0$. Further, there exists a function ${\mathscr R}_f:[a,b]^2_<\rightarrow X$ such that
\begin{align}
\label{spezzamento_young_int}
{\mathscr I}_f(s,t):=\lim_{|\Pi(s,t)|\rightarrow 0}\sum_{i=0}^{n-1}f(t_i)(x(t_{i+1})-x(t_i))=f(s)(x(t)-x(s))+{\mathscr R}_f(s,t)
\end{align}
for each $(s,t)\in[a,b]_<^2$, and
\begin{align}
\label{stima_resto_Rf}
\|{\mathscr R}_f\|_{\eta+\alpha|0,[a,b]}\leq \frac{1}{1-2^{-(\eta+\alpha-1)}}\|\delta_1f\|_{\alpha|0,[a,b]}\|x\|_{C^{\eta}([a,b])}.
\end{align}
In particular,
\begin{align}
\label{stima_int_young_nosemi}
\|{\mathscr I}_f\|_{\eta|0,[a,b]}\leq \bigg (\|f\|_0+\frac{{(b-a)^{\alpha}}}{1-2^{-(\eta+\alpha-1)}}\|\delta_1f\|_{\alpha|0,[a,b]}\bigg )\|x\|_{C^{\eta}([a,b])}.
\end{align}
\end{theorem}

\begin{remark}
\label{rem-davide}
{\rm For each $s, \tau,t\in [a,b]$, with $s<\tau<t$, it holds that
\begin{align}
\label{spezzamento_int_noconv}
\mathscr I_f(s,t)=\mathscr I_f(s,\tau)+\mathscr I_f(\tau,t).
\end{align}
To check this formula it suffices to choose a family of partitions $\Pi(s,t)$ such that $\tau\in \Pi(s,t)$ and letting $|\Pi(s,t)|$ tend to $0$. As a byproduct, if we set $\Phi(t):=\mathscr I_f(a,t)$, $t\in(a,b]$, we deduce that $(\delta_1\Phi)(s,t)=\mathscr I_f(s,t)$. Indeed, from \eqref{spezzamento_int_noconv} we infer
\begin{align*}
(\delta_1\Phi)(s,t)=\mathscr I_f(a,t)-\mathscr I_f(a,s)=\mathscr I_f(s,t).
\end{align*}
Moreover, $\Phi$ is the unique function such that $\Phi(a)=0$ and
\begin{eqnarray*}
|(\delta_1\Phi)(t,s)-f(s)(\delta_1x)(t,s)|\le c|t-s|^{\alpha+\eta}
\end{eqnarray*}
for every $(s,t)\in [a,b]_{<}^2$ and some positive constant $c$.}
\end{remark}

The above result reports the construction of the ``classical'' Young integral. The following one, proved in \cite[Sections 3 \& 4]{GT10}, accounts the construction of Young type convolutions with the semigroup $(S(t))_{t\ge 0}$.

\begin{theorem}
\label{teo-2.4}
For each $f\in \hat C_k([a,b];X_\beta)$, such that $\beta\in [0,2)$ and $\eta+k>1$, the limit
\begin{align}
\label{somme_di_riemann}
\lim_{|\Pi(s,t)|\to 0}\sum_{i=0}^{n-1} S(t-t_{i})f(t_i)(x(t_{i+1})-x({t_i}))
\end{align}
exists in $X$ for every $(s,t)\in [a,b]^2_{<}$.
Further, there exists a function ${\mathscr R}_{Sf}:[a,b]_<^2\rightarrow X$ such that
\begin{align*}
{\mathscr I}_{Sf}(s,t):= & \lim_{|\Pi(s,t)|\rightarrow 0}\sum_{i=0}^{n-1} S(t-t_{i})f(t_i)(x(t_{i+1})-x({t_i})) \\
= & S(t-s)f(s)(x(t)-x(s))+{\mathscr R}_{Sf}(s,t),
\end{align*}
for each $(s,t)\in[a,b]_<^2$, and for each $\varepsilon\in[0,1)$ there exists a positive constant $c=c(\eta+\alpha,\varepsilon)$ such that
\begin{align}
\label{stima_sewing_map_conv}
\|{\mathscr R}_{Sf}\|_{\eta+k-\varepsilon|\beta+\varepsilon,[a,b]}
\leq c\|\hat\delta_1f\|_{k|\beta,[a,b]}\|x\|_{C^{\eta}([a,b])}.
\end{align}
In particular,
\begin{align*}
\|{\mathscr I}_{Sf}\|_{\eta|\beta,[a,b]}
\leq &M_{0,\beta,b}\|f\|_\beta\|x\|_\eta+\|{\mathscr R}_{Sf}\|_{\eta|\beta,[a,b]}\\
\leq &\big (M_{0,\beta,b}\|f\|_\beta+{c(k,a,b)}\|\hat\delta_1f\|_{k|\beta,[a,b]}\big )\|x\|_{C^{\eta}([a,b])}.
\end{align*}
\end{theorem}

\begin{remark}
{\rm Actually, in \cite{GT10}, Theorem \ref{teo-2.4} has been proved assuming that $X_{\beta}=D((-A)^{\beta})$.
A direct inspection of the proof of \cite[Theorem 4.1(2)]{GT10} shows that the assertion holds true also under our assumptions, since
estimates \eqref{stime_smgr} allow us to repeat verbatim the same arguments in the quoted paper.}
\end{remark}

\begin{remark}
{\rm Clearly, when $x\in C^1([a,b])$
\begin{itemize}
\item
the limit in \eqref{somme_di_riemann-0} coincides with the Riemann-Stieltjes integral over the interval $[s,t]$ of the function $f$ with respect to the function $x$;
\item
the limit in \eqref{somme_di_riemann} coincides with the Riemann-Stieltjes integral of the function $S(t-\cdot)f$ with respect to the function $x$ over the interval $[s,t]$
\end{itemize}
for every $(s,t)\in [a,b]^2_{<}$.}
\end{remark}

The previous remark yields the following definition (see \cite{GT10}).

\begin{definition}
For every  $f\in C^{\alpha}([0,T])$ $(\alpha\in (1-\eta,1))$, ${\mathscr I}_f(s,t)$ is the Young integral of $f$ in $[s,t]$ for every $(s,t)\in [a,b]^2_{<}$ and it is denoted by
\begin{eqnarray*}
\int_s^tf(u)dx(u).
\end{eqnarray*}
Similarly, for every $f\in \hat C_k([a,b];X_\beta)$, with $k\in (1-\eta,1)$ and $\beta\in [0,2)$, ${\mathscr I}_{Sf}(s,t)$ is the Young integral of the function $S(t-\cdot)f$ with respect to $x$ in $[s,t]$ for every $(s,t)\in [a,b]^2_{<}$ and it is denoted by
 \begin{align}
\int_s^tS(t-u)f(u)dx(u).
\label{young_integral_def}
\end{align}
\end{definition}

For further use, we prove a slight extension of the estimate in \cite[Theorem 4.1(2)]{GT10}.

\begin{lemma}
\label{thm:young_integral}
Let $f$ be a function in $\hat C_k([a,b];X_{\beta})\cap C([a,b];X_{\beta_1})$ and assume that $k\in (1-\eta,1)$ and $\beta,\beta_1\in [0,2)$. Then, for every $r\in[k,1)$ the function ${\mathscr I}_{Sf}$ belongs to $C_{\eta+k-r}([a,b]_<^2;X_{\nu_r})$, where $\nu_r:=\min\{r+\beta,r+\beta_1-k\}$. Further,
\begin{align}
\label{integrale_stima_migliore}
\|{\mathscr I}_{Sf}\|_{\eta+k-r|\nu_r,[a,b]}
\le C_{\beta_1,\eta,r,k}\|x\|_{C^{\eta}([a,b])}(\|\hat\delta_1f\|_{k|\beta,[a,b]}+\|f\|_{\beta_1,[a,b]})
\end{align}
for every $r\in[k,1)$.
\end{lemma}

\begin{proof}
From Theorem \ref{teo-2.4} it follows that $\mathscr If$ is well-defined as Young integral and
\begin{align*}
({\mathscr I}_{Sf})(s,t)
=(x(t)-x(s))S(t-s)f(s)+{\mathscr R}_{Sf}(s,t), \qquad\;\, (s,t)\in[a,b]_<^2.
\end{align*}
Using condition \eqref{stime_smgr}(a), we get
\begin{align}
|(x(t)-x(s))S(t-s)f(s)|_{\gamma+\beta_1}
\leq & [x]_{C^{\eta}([a,b])}|t-s|^\eta|S(t-s)f|_{\gamma+\beta_1}\notag\\
\leq &M_{\beta_1,\gamma+\beta_1,b}[x]_{C^{\eta}([a,b])}\|f\|_{\beta_1,[a,b]}|t-s|^{\eta-\gamma}
\label{stima_incrementata_X1}
\end{align}
for each $(s,t)\in [a,b]^2_{<}$, $\gamma\in [0,\eta)$.

Now, we fix $r\in[ k,1)$ and take $\gamma=r-k$. Since $\eta+k>1$ it follows that $\gamma<1-k<\eta$ and $\eta-\gamma=\eta+k-r$. From \eqref{stima_sewing_map_conv} and \eqref{stima_incrementata_X1} we conclude that ${\mathscr I}_{Sf}\in C_{\eta+k-r}([a,b]_<^2;X_{\nu_r})$, where $\nu_r:=\min\{r+\beta, r+\beta_1-k\}$, and estimate \eqref{integrale_stima_migliore} follows.
\end{proof}

\begin{remark}
{\rm From the definition of the Young integral it follows that if $x,x_1,x_2\in C^\eta([a,b])$ and $f,f_1,f_2\in \hat C_k([a,b];X_\beta)$, for some $\eta\in (1/2,1)$, $k\in (1-\eta,1)$ and $\beta\in [0,2)$, then
\begin{align}
\label{young_integral_linearity}
\int_s^tS(t-u)f(u)d(x_1+x_2)(u)=\int_s^tS(t-u)f(u)dx_1(u)+\int_s^tS(t-u)f(u)dx_2(u)
\end{align}
and
\begin{align}
\label{young_integral_linearity_1}
\int_s^tS(t-u)(f_1(u)\!+\!f_2(u))dx(u)=\int_s^tS(t-u)f_1(u)dx(u)\!+\!\int_s^tS(t-u)f_2(u)dx(u)
\end{align}
for every $(s,t)\in [a,b]^2_{<}$}.
\end{remark}

Now, we prove that the Young integral \eqref{young_integral_def} can be split into the sum of two terms.

\begin{lemma}
\label{lem:div_int}
For every $f\in \hat C_k([a,b];X_{\beta})$, with $\beta\in [0,2)$ and $k\in (1-\eta,1)$, every
$(s,t)\in [a,b]^2_{<}$ and $\tau\in[s,t]$, it holds that
\begin{align*}
\int_s^tS(t-r)f(r)dx(r)=S(t-\tau)\int_s^\tau S(\tau-r)f(r)dx(r)+\int_\tau^tS(t-r)f(r)dx(r).
\end{align*}
\end{lemma}

\begin{proof}
Fix $f$ as in the statement. If $\tau=s$ or $\tau=t$, then the assertion is straightforward. So, let us assume that $\tau\in(s,t)$.
By the definition of the Young integral, we can determine a sequence $\{\Pi_n(s,t)\}$ of partitions of the interval $[s,t]$ such that
\begin{align}
\label{young_integral_riemann_sum}
\int_s^tS(t-r)f(r)dx(r)
=\lim_{n\to +\infty}\sum_{i=0}^{k_n-1}S(t-t_i^n)f(t_i^n)(x(t_{i+1}^n)-x(t_i^n)),
\end{align}
where we have set $\Pi_n(s,t)=\{a=t_0^n<\cdots<t_{k_n}^n=b\}$. Without loss of generality, we can assume that the point $\tau$ belongs to the partition
$\Pi_n$ for every $n\in\N$. For each $n\in\N$, we denote by $h_n$ the index such that $t_{h_n}^n=\tau$. Then, we split
\begin{align*}
&\sum_{i=0}^{k_n-1}S(t-t_i^n)f(t_i^n)(x(t_{i+1}^n)-x(t_i^n))\\
=&\sum_{i=0}^{h_n-1}S(t-t_i^n)f(t_i^n)(x(t_{i+1}^n)-x(t_i^n))\\
&+\sum_{i=h_n}^{k_n-1}S(t-t_i^n)f(t_i^n)(x(t_{i+1}^n)-x(t_i^n)).
\end{align*}
Note that the second term in the right-hand side of the previous formula converges to
$\displaystyle\int_{\tau}^tS(t-u)f(u)dx(u)$ as $n$ tends to $+\infty$. On the other hand, applying the semigroup property we infer that
\begin{align*}
S(t-t_i^n)f(t_i^n)(x(t_{i+1}^n)-x(t_i^n))
= S(t-\tau)S(\tau-t_i^n)f(t_i^n)(x(t_{i+1}^n)-x(t_i^n))
\end{align*}
for every $i=0,\ldots,n_\tau$.
Therefore,
\begin{align}
&\lim_{n\to +\infty}\sum_{i=0}^{h_n-1}S(t-t_i^n)f(t_i^n)(x(t_{i+1}^n)-x(t_i^n))\notag\\
= & \lim_{n\to +\infty}S(t-\tau)\sum_{i=0}^{h_n-1}S(\tau-t_i^n)f(t_i^n)(x(t_{i+1}^n)-x(t_i^n)) \notag \\
= & S(t-\tau)\lim_{n\to +\infty}\sum_{i=0}^{h_n-1}S(\tau-t_i^n)f(t_i^n)(x(t_{i+1}^n)-x(t_i^n)) \notag \\
= & S(t-\tau)\int_s^\tau S(\tau-u)f(u)dx(u).
\label{young_integral_riemann_sum3}
\end{align}
From \eqref{young_integral_riemann_sum}-\eqref{young_integral_riemann_sum3} the assertion follows easily.
\end{proof}

\begin{corollary}
\label{integrale_ancorato}
For every $f\in \hat C_k([a,b];X_{\beta})$, with $k+\eta>1$ and $\beta\in [0,2)$, it holds that
\begin{align*}
(\hat \delta_1{\mathscr I}_{Sf}(a,\cdot))(s,t)={\mathscr I}_{Sf}(s,t)=\int_s^tS(t-r)f(r)dx(r), \qquad\;\,(s,t)\in [a,b]^2_{<}.
\end{align*}
\end{corollary}

\begin{proof}
From the definition of $\hat\delta_1$ and of ${\mathscr I}_{Sf}$ it follows that
\begin{eqnarray}
\label{delta_int_ancorato}
(\hat\delta_1{\mathscr I}_{Sf}(a,\cdot))(s,t)=\int_a^tS(t-r)f(r)dx(r)-S(t-s)\int_a^sS(s-r)f(r)dx(r)
\end{eqnarray}
for every $(s,t)\in [a,b]^2_{<}$. Applying Lemma \ref{lem:div_int} with $s=a$ and $\tau=s$ we infer that
\begin{align*}
\int_a^tS(t-r)f(r)dx(r)=S(t-s)\int_a^s S(s-r)f(r)dx(r)+\int_s^tS(t-r)f(r)dx(r),
\end{align*}
which combined with \eqref{delta_int_ancorato} yields the assertion.
\end{proof}

\section{Mild solution and smoothness}
\label{sect-3}
We consider the following assumptions on the nonlinear term $\sigma$.
\begin{hypothesis}
\label{ip_nonlineare}
The function $\sigma:X\to X$ is Fr\'echet differentiable with bounded and locally Lipschitz continuous Fr\'echet derivative. Moreover, the restriction of $\sigma$ to $X_{\alpha}$ maps this space into itself for some $\alpha\in (0,1)$ such that $\alpha+\eta>1$, it is locally Lipschitz continuous and there exists a positive constant  $L_\sigma^\alpha$ such that
\begin{align}
\label{ip_sigma}
|\sigma(x)|_\alpha\leq L_\sigma^{\alpha}(1+|x|_\alpha),\qquad\;\,x\in X_{\alpha}.
\end{align}
\end{hypothesis}

Hereafter, we assume that Hypothesis \ref{ip:young_path} with $a=0$ and $b=T>0$ and Hypothesis \ref{ip_nonlineare} hold true.

We consider the following nonlinear Young equation
\begin{align}
\label{omo_equation}
\left\{
\begin{array}{ll}
dy(t)=Ay(t)dt+\sigma(y(t))dx(t), &t\in(0,T],\\[1mm]
y(0)=\psi.
\end{array}
\right.
\end{align}
and we are interested in its mild and classical solutions, where by mild solution we mean a function $y:[0,T]\to X$ such that $\sigma(y)\in \hat C_{\alpha}([0,T];X)$, $\eta+\alpha>1$ and
\begin{equation}
y(t)=S(t)\psi+({\mathscr I}_{S\sigma(y)})(0,t),\qquad\;\,t\in [0,T].
\label{reg_sol-aaa}
\end{equation}

\begin{theorem}
\label{es_mild_sol_gub}
Let Hypotheses $\ref{hyp-main}$, $\ref{ip:young_path}$ and $\ref{ip_nonlineare}$ be satisfied, with $[a,b]=[0,T]$. Then, for every $\psi\in X_\alpha$ such that $\alpha\in (0,1/2)$ and $\eta+\alpha>1$, there exists a unique mild solution $y\in \hat C_\alpha([0,T];X_\alpha)$ to the stochastic equation \eqref{omo_equation}.
The solution $y$ is actually smoother since for every $a\in(0,T)$ and $\gamma\in [\eta+\alpha-1,\eta+\alpha)$, $y$ belongs to $\hat C_{\eta+\alpha-\gamma}([a,T];X_{\gamma})$.
Moreover, for every $\mu\in[0,\eta+\alpha-1)$ and $\varepsilon>0$ there exists a positive constant $c=c(\varepsilon,\mu)$ such that
\begin{eqnarray}
\label{stima_lambda_mu}
|y(t)|_{1+\mu}\leq ct^{\eta+\alpha-2-\mu-\varepsilon},\qquad\;\,t\in(0,T].
\end{eqnarray}
In particular, $y$ is a classical solution to \eqref{omo_equation}, i.e., $y(t)\in D(A)$ for every $t\in (0,T]$ and
it belongs to $C_{\eta-\beta}([a,T];X_{\alpha+\beta})$ for every $a\in (0,T)$ and $\beta\in [0,\eta)$.
\end{theorem}

The proof follows the lines of \cite[Theorem 4.3]{GT10}, but our assumptions are weaker. In particular, in \cite{GT10} the authors assume that $\eta>2\alpha$, while we do not need this condition.

Before proving Theorem \ref{es_mild_sol_gub}, we state the following lemma, which is a straightforward consequence of Lemma \ref{lem:div_int}.

\begin{lemma}
Suppose that $y$ is a mild solution to \eqref{omo_equation}. Then, for every $\tau\in[0,T]$ it holds that
\begin{align}
\label{scomp_soluzione}
y(t)=S(t-\tau)y(\tau)+\int_\tau^t S(t-r)(\sigma(y(r)))dx(r), \qquad\;\, t\in[\tau,T].
\end{align}
\end{lemma}

\begin{proof}[Proof of Theorem $\ref{es_mild_sol_gub}$]
We split the proof into some steps.

{\em Step 1.} Here, we prove an apriori estimate. Namely, we show that if $y\in \hat C_\alpha([0,T];X_\alpha)$ is a mild solution to \eqref{omo_equation}, then there exists a positive constant $\mathfrak R$, which depends only on $\psi$, $T$, $\alpha$, $x$, $\eta$ and $\sigma$, such that
\begin{align}
\label{aprioir_est_sol}
\|y\|_{\alpha,\alpha,[0,T]}\leq \mathfrak R.
\end{align}

Let us fix $a,b\in [0,T]$, with $a<b$.
Taking Corollary \ref{integrale_ancorato} into account, it is easy to check that $(\hat\delta_1y)(s,t)=({\mathscr I}_{S\sigma(y)})(s,t)$ for every $(s,t)\in [0,T]^2_{<}$. Hence, to estimate $\|\hat\delta_1y\|_{\alpha|\alpha,[a,b]}$ we can take advantage of Lemma \ref{thm:young_integral}.
For this purpose, let us prove that $\sigma(y)$ belongs to $\hat C_\alpha([a,b];X)\cap C([a,b];X_{\alpha})$. The condition $\sigma(y)\in C([a,b];X_\alpha)$ follows immediately from \eqref{ip_sigma}, which also shows that
\begin{align}
\|\sigma(y)\|_{\alpha,[a,b]}\leq L_\sigma^{\alpha}(1+\|y\|_{\alpha,[a,b]}).
\label{stima_sigma_2}
\end{align}
Further, we note that the function $\hat\delta_1(\sigma(y))$ is continuous in $[0,T]$ with values in $X$. Indeed, fix $(t_0,s_0)\in [a,b]_{<}^2$. Then,
\begin{align}
&|(\hat\delta_1(\sigma(y))(t,s)-(\hat\delta_1(\sigma(y)))(t_0,s_0)|_0\notag\\
\le &|(\sigma(y(t))-(\sigma(y(t_0))|_0+|S(t-s)\sigma(y(s))-S(t_0-s_0)\sigma(y(s_0))|_0\notag\\
\le &L|y(t)-y(t_0)|_0+\|S(t-s)\|_{{\mathscr L}(X_0)}|\sigma(y(s))-\sigma(y(s_0))|_0 \notag \\
&+|(S(t-s)-S(t_0-s_0))\sigma(y(s_0))|_0\notag\\
\le &L|y(t)-y(t_0)|_0+L M_{0,0}|y(s)-y(s_0)|_{0}
+2C_{\alpha,0}
|\sigma(y(s_0))|_{\alpha}|t-t_0|^{\alpha}
\label{video}
\end{align}
for every $(t,s)\in [0,T]^2_{<}$, where $L$ denotes the Lipschitz constant of $\sigma$ on $X_0$, and the last side of the previous chain of inequalities vanishes as $(t,s)$ tends to $(t_0,s_0)$. Next, we split
\begin{align*}
(\hat\delta_1\sigma(y))(s,t)=(\delta_1\sigma(y))(s,t)-\mathfrak{a}(s,t)\sigma(y(s)), \qquad\;\, (s,t)\in [0,T]^2_{<}.
\end{align*}
and estimate separately the two terms. As far as the first one is considered, we observe that
\begin{align}
\label{stima_delta}
|(\delta_1\sigma(y))(s,t)|_0
= & |\sigma(y(t))-\sigma(y(s))|_0\notag\\
\le &L_{\sigma}|y(t)-y(s)|_0\notag\\
\leq &L_\sigma(|(\hat\delta_1 y)(s,t)|_0+|\mathfrak{a}(s,t)y(s)|_0)\notag \\
\leq & L_\sigma(1+C_{\alpha,0,T})\|y\|_{\alpha,\alpha,[a,b]}|t-s|^\alpha
\end{align}
for every $(s,t)\in [a,b]^2_{<}$,
where $C_{\alpha,0,T}$ is the constant in condition \eqref{stime_smgr}$(b)$ and $L_{\sigma}$ denotes the Lipschitz constant of the function $\sigma$.
As far as the term $\mathfrak{a}(s,t)\sigma(y(s))$ is concerned, we use \eqref{ip_sigma} to  estimate
\begin{align*}
|\mathfrak{a}(s,t)\sigma(y(s))|_0
\le C_{\alpha,0,T}|\sigma(y(s))|_\alpha|t-s|^\alpha
\leq C_{\alpha,0,T}L_{\sigma}^{\alpha}(1+\|y\|_{\alpha,\alpha,[a,b]})|t-s|^\alpha
\end{align*}
for every $(s,t)\in [a,b]^2_{<}$. We have so proved that $\sigma(y)\in\hat C_{\alpha}([a,b];X)$ and
\begin{align}
\label{stima_sigma_1}
\|(\hat\delta_1\sigma)(y)\|_{\alpha|0,[a,b]}\leq (L_\sigma+L_{\sigma}^{\alpha})(1+C_{\alpha,0,T})(1+\|y\|_{\alpha,\alpha,[a,b]}).
\end{align}

Thus, we can apply Lemma \ref{thm:young_integral} as claimed, with $k=\beta_1=\alpha$ and $\beta=0$, to infer that
${\mathscr I}_{S\sigma(y)}$ belongs to $C_{\eta+\alpha-r}([a,b]_<^2;X_r)$ for every $r\in[\alpha,1)$ and
\begin{align}
\|{\mathscr I}_{S\sigma(y)}\|_{\eta+\alpha-r|r,[a,b]}\le &C_{\alpha,\eta,r,\alpha}\|x\|_{C^{\eta}([0,T])}(\|(\hat\delta_1\sigma)(y)\|_{\alpha|0,[a,b]}+\|\sigma(y)\|_{\alpha,[a,b]})\notag\\
\le & C_{\alpha,\eta,r,\alpha}\|x\|_{C^{\eta}([0,T])}(L_{\sigma}+L_{\sigma}^{\alpha})(2+C_{\alpha,0,T})(1+\|y\|_{\alpha,\alpha,[a,b]}).
\label{stima-utile}
\end{align}
Since $\alpha<1/2<\eta$, it follows that
$\|{\mathscr I}_{S\sigma(y)}\|_{\alpha|\alpha,[a,b]}\leq (b-a)^{\eta-\alpha}\|{\mathscr I}_{S\sigma(y)}\|_{\eta|\alpha,[a,b]}$, so that,
applying \eqref{stima-utile} with $r=\alpha$, we conclude that
\begin{align}
\|\hat\delta_1y\|_{\alpha|\alpha,[a,b]}=\|{\mathscr I}_{S\sigma(y)}\|_{\alpha|\alpha,[a,b]}
\leq \mathfrak C (b-a)^{\eta-\alpha} \|x\|_\eta (1+\|y\|_{\alpha,\alpha,[a,b]}),
\label{puntino}
\end{align}
where $\mathfrak C:=C_{\alpha,\eta,r,\alpha}(L_\sigma+L_\sigma^\alpha)(2+C_{\alpha,0,T})$. Further, from \eqref{scomp_soluzione} with $\tau=a$, $t\in[a,b]$ and Corollary \ref{integrale_ancorato}, we get
\begin{align}
\|y\|_{\alpha,[a,b]}
\le & M_{\alpha,\alpha,b}|y(a)|_\alpha+\|(\hat\delta_1y)(a,\cdot)\|_{\alpha,[a,b]}\notag\\
\leq & M_{\alpha,\alpha,T}|y(a)|_\alpha+(b-a)^{\alpha}\|\hat\delta_1y\|_{\alpha|\alpha,[a,b]}\notag\\
\leq & M_{\alpha,\alpha,T}|y(a)|_\alpha+\mathfrak C(b-a)^\eta   \|x\|_\eta(1+\|y\|_{\alpha,\alpha,[a,b]}).
\label{duepuntini}
\end{align}
Taking \eqref{puntino} and \eqref{duepuntini} into account, this gives
\begin{align}
\|y\|_{\alpha,\alpha,[a,b]}
\leq & M_{\alpha,\alpha,T}|y(a)|_\alpha+\mathfrak C(b-a)^{\eta-\alpha}(1+(b-a)^{\alpha}) \|x\|_\eta (1+\|y\|_{\alpha,\alpha,[a,b]})\notag\\
\leq & M_{\alpha,\alpha,T}|y(a)|_\alpha+\mathfrak C(b-a)^{\eta-\alpha}(1+T^{\alpha}) \|x\|_\eta (1+\|y\|_{\alpha,\alpha,[a,b]}).
\label{cadabra-3}
\end{align}
Let us set
\begin{align*}
\overline T=\bigg (\frac{1}{2\mathfrak C(1+T^\alpha)\|x\|_\eta}\bigg )^{\frac{1}{\eta-\alpha}}.
\end{align*}
If $b-a\le \overline T$, then we get
\begin{align}
\|y\|_{\alpha,\alpha,[a,b]}
\leq 2M_{\alpha,\alpha,T}|y(a)|_\alpha+1.
\label{idea}
\end{align}
Now, we introduce the function $\phi:(0,\infty)\to (0,\infty)$, defined by
$\phi(r)=2M_{\alpha,\alpha,T} r+1$ for every $r>0$ and split
\begin{align*}
[0,T]=\bigcup_{n=0}^{N-1}[t_n,t_{n+1}],
\end{align*}
where $0=t_0<t_1<t_2<\ldots<t_N=T$ and $t_{n+1}-t_n\leq \overline T$ for every $n=0,\ldots,N-1$. From \eqref{idea} it follows that
\begin{align}
\|y\|_{\alpha,[t_n,t_{n+1}]}
\leq \phi(|y(t_n)|_\alpha)\leq \phi^{n+1}(|\psi|_\alpha),
\label{ibra}
\end{align}
for every $n=0,\ldots,N-1$, where $\phi^k$ denotes the composition of $\phi$ with itself $k$ times.
Since $\phi(r)>r$ for every $r>0$, from \eqref{ibra} we conclude that
\begin{align}
\|y\|_{\alpha,[0,T]}
\leq \phi^N(|\psi|_\alpha).
\label{cadabra}
\end{align}
In particular, for each interval $[s,t]\subset[0,T]$ whose length is less than or equal to $\overline T$ we get
\begin{align*}
\|y\|_{\alpha,\alpha,[s,t]}
\leq 2M_{\alpha,\alpha,T}\phi^N(|y(s)|_\alpha)+1\le \phi^{N+1}(|\psi|_{\alpha}).
\end{align*}
Now we are able to estimate $\|\hat\delta_1y\|_{\alpha|\alpha,[0,T]}$. We stress that, if $|t-s|\leq \overline T$, then from \eqref{idea} we get
\begin{align}
{|(\hat\delta_1y)(s,t)|_\alpha}
\leq \phi^{N+1}(|\psi|_\alpha)|t-s|^\alpha,
\label{cadabra-1}
\end{align}
and if $|t-s|>\overline T$ then
\begin{align}
\frac{|(\hat\delta_1y)(s,t)|_\alpha}{|t-s|^\alpha}
\leq \frac{|y(t)-S(t-s)y(s)|_{\alpha}}{\overline T^{\alpha}}
\leq \frac{(1+M_{\alpha,\alpha,T})\phi^N(|\psi|_{\alpha})}{\overline T^\alpha}.
\label{cadabra-2}
\end{align}

From \eqref{cadabra}, \eqref{cadabra-1} and \eqref{cadabra-2} it follows that
\begin{align*}
\|y\|_{\alpha,\alpha,[0,T]}\le \phi^N(|\psi|_{\alpha})+\max\{\phi^{N+1}(|\psi|_{\alpha}),\overline T^{-\alpha}(1+M_{\alpha,\alpha,T})\phi^N(|\psi|_{\alpha})\}=:{\mathfrak R}.
\end{align*}

{\em Step 2.} Here, we prove that there exists a unique mild solution to the equation \eqref{omo_equation}.
For this purpose, we introduce the operator $\Gamma_1:\hat C_\alpha([0,T_*];X_\alpha)\rightarrow \hat C_\alpha([0,T_*];X_\alpha)$, defined by
$(\Gamma_1(y))(t)=S(t)\psi+{\mathscr I}_{S\sigma(y)}(0,t)$ for every $t\in (0,T_*]$ and $(\Gamma_1(y))(0)=\psi$,
where $T_*\in(0,T]$ has to be properly chosen later on.
We are going to prove that $\Gamma_1$ is a contraction in $\mathcal B=\{y\in\hat C_\alpha([0,T_*];X_\alpha):\|y\|_{\alpha,\alpha,[0,T_*]}\leq 2M_{\alpha,\alpha,T}\mathfrak R\}$. To begin with, we fix $y\in \mathcal B$ and observe that $\hat\delta_1\Gamma_1(y)={\mathscr I}_{S\sigma(y)}$. Hence, from \eqref{cadabra-3}, we can estimate
\begin{align}
\|\Gamma_1(y)\|_{\alpha,\alpha,[0,T_*]}
\leq & M_{\alpha,\alpha,T}|\psi|_\alpha+\mathfrak{C}T_*^{\eta-\alpha}(1+T_*^{\alpha})\|x\|_{C^{\eta}([0,T])}(1+\|y\|_{\alpha,\alpha,[0,T_*]})\notag \\
\leq & M_{\alpha,\alpha,T}\mathfrak{R}+\mathfrak{C}T_*^{\eta-\alpha}(1+T^\alpha)\|x\|_{C^{\eta}([0,T])}(1+2M_{\alpha,\alpha,T}\mathfrak R).
\label{punto_fisso_stima_2}
\end{align}
We now choose $T_*\le T$ such that $\mathfrak{C}T_*^{\eta-\alpha}(1+T^\alpha)\|x\|_{C^{\eta}([0,T])}(1+2M_{\alpha,\alpha,T}\mathfrak R)\le M_{\alpha,\alpha,T}{\mathfrak R}$. With this choice of $T_*$, we conclude that $\Gamma_1(y)$ belongs to ${\mathcal B}$.

Let us prove that $\Gamma_1$ is a $1/2$-contraction. Fix $y_1,y_2\in \mathcal B$. The linearity of the Young integral gives
$(\Gamma_1(y_1))(t)-(\Gamma_1(y_2))(t)={\mathscr I}_{S(\sigma(y_1)-\sigma(y_2))}(0,t)$ for every $t\in[0,T_*]$,
so that we can estimate
\begin{align}
\label{controllo_stima}
\|\Gamma_1(y_1)-\Gamma_1(y_2)\|_\alpha\le T_*^{\eta}\|{\mathscr I}_{S(\sigma(y_1)-\sigma(y_2))}\|_{\eta|\alpha,[0,T_*]}
\end{align}
and, as in Step 1 (see the first inequality in \eqref{stima-utile}),
\begin{align}
&\|\Gamma_1(y_1)-\Gamma_1(y_2)\|_{\eta|\alpha,[0,T_*]}\notag\\
\le & C_{\alpha,\eta}\|x\|_{C^{\eta}([0,T])}(\|\hat\delta_1(\sigma(y_1)-\sigma(y_2))\|_{\alpha|0,[0,T_*]}+\|\sigma(y_1)-\sigma(y_2)\|_{\alpha,[0,T_*]}).
\label{stima_Mario}
\end{align}
We set $R:=2M_{\alpha,\alpha,T}\mathfrak{R}\geq\max\{\|y_1\|_{\alpha,[0,T_*]},\|y_2\|_{\alpha,[0,T_*]}\}$ and note that
\begin{align}
|\mathfrak{a}(s,t)(\sigma(y_1(s))-\sigma(y_2(s)))|_0
\leq & C_{\alpha,0,T}|t-s|^\alpha|\sigma(y_1(s))-\sigma(y_2(s))|_{\alpha}\notag\\
\leq & C_{\alpha,0,T}L_\sigma^{\alpha,R}|t-s|^\alpha\|y_1-y_2\|_{\alpha,[0,T_*]},
\label{approx_stima_1_1}
\end{align}
where $L_\sigma^{\alpha,r}$ denotes the Lipschitz constant of the restriction of $\sigma$ to the ball $B(0,r)\subset X_{\alpha}$ and we have used the condition
\eqref{stime_smgr}(b). Further, by taking advantage of the smoothness of $\sigma$ we get
\begin{align}
&(\delta_1(\sigma(y_1)-\sigma(y_2))(s,t))\notag\\
= & \sigma(y_1(s)+(\delta_1y_1)(s,t))-\sigma(y_1(s))-\sigma(y_2(s)+(\delta_1y_1)(s,t)) +\sigma(y_2(s))\notag\\
&+\sigma(y_2(s)+(\delta_1y_1)(s,t))-\sigma(y_2(s)+(\delta_1y_2)(s,t))\notag\\
= & \int_0^1\langle \sigma'(y_1(s)+r(\delta_1y_1)(s,t))-\sigma'(y_2(s)+r(\delta_1y_1)(s,t)),\delta_1y_1(s,t)\rangle dr\notag\\
& +\sigma(y_2(s)+(\delta_1y_1)(s,t))-\sigma(y_2(s)+(\delta_1y_2)(s,t)).
\label{A}
\end{align}
Since for every $s,t\in[0,\overline T]$, with $s<t$, and $r\in(0,1)$, it holds that
\begin{align*}
|y_1(s)+r\delta_1y_1(s,t)|_\alpha\vee
|y_2(s)+r\delta_1y_1(s,t)|_\alpha
\leq 3R,
\end{align*}
and recalling that $R\ge 1$, it follows that
\begin{align}
&|\delta_1(\sigma(y_1)-\sigma(y_2))(s,t)|_0\notag\\
\leq & K_{\sigma'}^R\|y_1-y_2\|_{0,[0,T_*]}|(\delta_1y_1)(s,t)|_0 +L_\sigma|(\delta_1(y_1- y_2))(s,t)|_0\notag \\
\leq & K_{\sigma'}^R\|y_1-y_2\|_{0,[0,T_*]}(|(\hat\delta_1y_1)(s,t)|_0+|\mathfrak{a}(s,t)y_1(s)|_0)\notag\\
& +L_{\sigma}(|(\hat\delta_1(y_1-y_2))(s,t)|_0+|\mathfrak{a}(s,t)(y_1(s)-y_2(s))|_0)\notag\\
\le &K_{\sigma'}^R(\|y_1\|_{\alpha,\alpha,[0,T_*]}+C_{\alpha,0,T} \|y_1\|_{\alpha,[0,T_*]})\|y_1-y_2\|_{0,[0,T_*]}|t-s|^\alpha\notag\\
&+L_{\sigma}(\|y_1-y_2\|_{\alpha,\alpha,[0,T_*]}\!+\!C_{\alpha,0,T} \|y_1-y_2\|_{\alpha,[0,T_*]})|t-s|^\alpha\notag\\
\le & (1+C_{\alpha,0,T})R(K_{\sigma'}^R+L_{\sigma})\|y_1-y_2\|_{\alpha,\alpha,[0,T_*]}|t-s|^\alpha,
\label{approx_stima_2_1}
\end{align}
where $K_{\sigma'}^R$ denotes the Lipschitz constant of the restriction of function $\sigma'$ to the ball $B(3K_{\alpha,0}R)\subset X$. As far as $\|\sigma(y_1)-\sigma(y_2)\|_{\alpha,[0,T_*]}$ in \eqref{stima_Mario} is concerned, it holds that
\begin{align}
\label{approx_stima_2_2}
|\sigma(y_1(t))-\sigma(y_2(t))|_\alpha
\le L_\sigma^{\alpha,R}\|y_1-y_2\|_{\alpha,\alpha,[0,T_*]}
\end{align}
for every $t\in[0,T_*]$.
From \eqref{stima_Mario}, \eqref{approx_stima_1_1}, \eqref{approx_stima_2_1} and \eqref{approx_stima_2_2} we get
\begin{align}
\label{stima_seconda_parte}
\|\Gamma_1(y_1)-\Gamma_1(y_2)\|_{\alpha|\alpha,[0,T_*]}
\le & T_*^{\eta-\alpha}\|{\mathscr I}_{S(\sigma(y_1)-\sigma(y_2))}\|_{\eta|\alpha,[0,T_*]}\notag\\
\le & cT_*^{\eta-\alpha}\|y_1-y_2\|_{\alpha,\alpha,[0,T_*]},
\end{align}
where $c$ is a positive constant which depends on $x, \alpha,R,\sigma,\eta$ but not on $T_*$ nor on $\psi$.

Based on \eqref{punto_fisso_stima_2} and \eqref{stima_seconda_parte}, we can now fix $T_*>0$ such that $\Gamma_1$ is a $1/2$-contraction in $\mathcal B$.
If $T_*=T$, then we are done. Otherwise, we use a standard procedure to extend the solution of the stochastic equation \eqref{omo_equation}: we introduce the operator $\Gamma_2$ defined by
\begin{align*}
(\Gamma_2(y))(t)=S(t-T_*)y_1(T_*)+{\mathscr I}_{S\sigma(y)}(T_*,t), \qquad T_*\leq t\leq T_{**}:=\min\{2T_*,T\},
\end{align*}
for every $y\in \mathcal B_2:=\{z\in \hat C_{\alpha}([T_*,T_{**}];X_\alpha):\|y\|_{\alpha,\alpha}\leq 2M_{\alpha,\alpha,T}\mathfrak R\}$. Since $y_1$ is a mild solution to \eqref{omo_equation}, from \eqref{aprioir_est_sol}, which clearly holds true also with $T_*<T$ and the same constant $\mathfrak R$,  it follows that $|y_1(T_*)|_\alpha\leq \mathfrak R$. Then, by the same computations as above we show that $\Gamma_2$ is a $1/2$-contraction in $\mathcal B_2$. Denote by $y_2$ its unique fixed point.
Thanks to Lemma \ref{lem:div_int}, the function $y$ defined by $y(t)=y_1(t)$ if $t\in [0,T_*]$ and $y(t)=y_2(t)$ if $t\in [T_*, T_{**}]$ is a mild solution to equation \eqref{omo_equation} in $[0,T_{**}]$. Obviously, if $T_{**}<T$, then we can repeat the same procedure and in a finite number of steps we extend $y$ to whole $[0,T]$. Estimate \eqref{aprioir_est_sol} yields also the uniqueness of the mild solution to equation \eqref{omo_equation}.

{\em Step 3.}
From the arguments in the first part of Step 1 (see \eqref{stima-utile}), we deduce that ${\mathscr I}_{S\sigma(y)}$ belongs to $C_{\eta+\alpha-r}([0,T]_<^2;X_r)$ for every $r\in[\alpha,1)$.
The smoothing properties of the semigroup $(S(t))_{t\ge 0}$ (see condition \eqref{stime_smgr}(a)), estimates \eqref{aprioir_est_sol} and \eqref{stima-utile} show that $y(t)\in X_r$ and
\begin{align}
\label{reg_sol}
|y(t)|_r\le & |S(t)\psi|_r+|{\mathscr I}_{S\sigma(y)}(0,t)|_r\notag\\
\le & M_{\alpha,r,T} t^{\alpha-r}|\psi|_{\alpha}+\|{\mathscr I}_{S\sigma(y)}\|_{\eta+\alpha-r|r,[0,T]}t^{\eta+\alpha-r}\notag\\
\le & c_1(1+T^{\eta})t^{\alpha-r}
\end{align}
for every $t\in (0,T]$ and some positive constant $c_1=c_1(\alpha,\eta,r,\|x\|_{C^{\eta}([0,T])},|\psi|_{\alpha},\sigma,\mathfrak R)$, which is a continuous function of $\|x\|_{C^{\eta}([0,T])}$ and $\mathfrak R$. Now, we observe that
$|y(t)-y(s)|_r\le |\hat\delta_1y(s,t)|_r+|\mathfrak{a}(s,t)y(s)|_r$.
Since $\hat\delta_1y={\mathscr I}_{S\sigma(y)}$, from \eqref{stima-utile} it follows that
\begin{align}
|\hat\delta_1y(s,t)|_r\le c_2(t-s)^{\eta+\alpha-r},\qquad\;\,(s,t)\in [0,T]^2_{<},
\label{questo}
\end{align}
where $c_2=c_2(\alpha,\eta,r,\|x\|_{C^{\eta}([0,T])},\psi,|\sigma|_{\alpha},{\mathfrak R})$ is a positive constant, which depends in a continuous way on $\|x\|_{C^{\eta}([0,T])}$ and
${\mathfrak R}$.
Moreover, using condition \eqref{stime_smgr}(b) and estimate \eqref{reg_sol} (with $r$ being replaced by $r+\beta$), we get
\begin{align}
|\mathfrak{a}(s,t)y(s)|_r
\le & C_{r+\beta,r,T}|t-s|^{\beta}|y(s)|_{r+\beta}\notag \\
\le & C_{r+\beta,r,T}{\tilde c_1}(1+T^{\eta})s^{\alpha-r-\beta}|t-s|^{\beta},
\label{quello}
\end{align}
where $\beta>0$ is such that $r+\beta<1$ (such $\beta$ exists since we are assuming $r\in [\alpha,1)$). From these two last estimates it follows immediately that $y\in C((0,T];X_r)$.
Moreover, for every $\varepsilon\in(0,T]$ and $r\in[\alpha,1)$, there exists a positive constant $c_3=c_3(\alpha,\eta,r,\|x\|_{C^{\eta}([0,T])},|\psi|_{\alpha},\sigma,\mathfrak R,T)$, which depends in a continuous way on $\|x\|_{C^{\eta}([0,T])}$ and on $\mathfrak R$, such that
\begin{align*}
\|y\|_{C([\varepsilon,T];X_r)}+\|\hat\delta_1y\|_{\eta+\alpha-r|r,[\varepsilon,T]}\leq c_3\varepsilon^{\alpha-r}.
\end{align*}

Next, we estimate $|(\hat\delta_1 (\sigma(y)))(s,t)|_{{\lambda}}$ when $\eta+\alpha-\lambda>1$, i.e., $\lambda\in[0,\eta+\alpha-1)$. As usually, we separately estimate $|(\delta_1(\sigma(y)))(s,t)|_{\lambda}$ and $|\mathfrak{a}(s,t)\sigma(y(s))|_{\lambda}$. Note that $\lambda<\alpha$ since $\eta<1$. We fix $\varepsilon>0$ and observe that the continuous embedding $X_{\alpha}\hookrightarrow X_{\lambda}$, \eqref{stima-utile} and \eqref{reg_sol} (with $r=2\alpha-\lambda$, which belongs to $[\alpha,1)$ since $\alpha<1/2$) give
\begin{align}
|(\delta_1(\sigma(y)))(s,t)|_{{\lambda}}
\le &K_{\alpha,\lambda}L_\sigma^{\alpha,\mathfrak R}(|(\hat\delta_1 y)(s,t)|_{{\alpha}}+|\mathfrak{a}(s,t)y(s)|_{{\alpha}})\notag\\
\le & K_{\alpha,\lambda}L_\sigma^{\alpha,\mathfrak R}(\|y\|_{\alpha,\alpha,[0,T]}|t-s|+C_{2\alpha-\lambda,\alpha,T}\varepsilon^{\lambda-\alpha}|t-s|^{\alpha-\lambda})\notag\\
\le & c_4\varepsilon^{\lambda-\alpha}|t-s|^{\alpha-\lambda}
\label{stima_a1}
\end{align}
for every $(s,t)\in [\varepsilon,T]^2_{<}$, where $c_4=c_4(\alpha,\eta,\|x\|_{C^{\eta}([0,T])},\mathfrak R,T,\lambda)$. Moreover,
\begin{align}
\label{stima_b}
|\mathfrak{a}(s,t)\sigma(y(s))|_{\lambda}
\le & C_{\alpha,\lambda,T}|\sigma(y(s))|_\alpha|t-s|^{\alpha-\lambda}\notag\\
\le & C_{\alpha,\lambda,T}L_\sigma^{\alpha}(1+|y(s)|_\alpha)|t-s|^{\alpha-\lambda}\notag\\
\le & C_{\alpha,\lambda,T}L_\sigma^{\alpha}(1+{\mathfrak R})|t-s|^{\alpha-\lambda}
\end{align}
for every $(s,t)\in [\varepsilon,T]^2_{<}$.
From \eqref{stima_a1} and \eqref{stima_b}, it follows that
\begin{eqnarray*}
\sup_{\varepsilon\le s<t\le T}\frac{|(\hat\delta_1(\sigma(y)))(s,t)|_{{\lambda}}}{|t-s|^{\alpha-\lambda}}\le c_4\varepsilon^{\lambda-\alpha}+C_{\alpha,\lambda,T}L_{\sigma}^{\alpha}(1+{\mathfrak R}).
\end{eqnarray*}
Moreover, arguing as in the proof of \eqref{video} we can show that
\begin{align*}
&|(\hat\delta_1(\sigma(y))(t,s)-(\hat\delta_1(\sigma(y))(t_0,s_0)|_{\lambda}\\
\le & L^\alpha K_{\alpha,\lambda}(|y(t)-y(t_0)|_{\alpha}+ M_{\alpha,\alpha}|y(s)-y(s_0)|_{\alpha})+2C_{\alpha,\lambda}|\sigma(y(s_0))|_{\alpha}|t-t_0|^{\alpha-\lambda},
\end{align*}
for every $(t_0,s_0), (t,s)\in [\varepsilon,T]^2_{<}$, where $L^\alpha$ denotes the Lipschitz constant of $\sigma$ on the subset $\{y\in X_\alpha:|y|_\alpha\leq \sup_{t\in[0,T]}{|y(t)|_\alpha}\}$ of $X_\alpha$, and conclude that $\hat \delta_1(\sigma(y))\in C_{\alpha-\lambda}([\varepsilon,T]_<^2;X_\lambda)$.
Further, $\sigma(y)$ belongs to $C([\varepsilon,T];X_\alpha)$. From Lemma \ref{thm:young_integral} with $k=\alpha-\lambda$, $\beta=\lambda$, $\beta_1=\alpha$ and $r=\gamma$, we infer that
${\mathscr I}_{S\sigma(y)}$ belongs to $C_{\eta+\alpha-\lambda-\gamma}([\varepsilon,T]_<^2;X_{\gamma+\lambda})$
for every $\gamma\in [\alpha-\lambda,1)$ and
\begin{align}
&\|{\mathscr I}_{S\sigma(y)}\|_{\eta+\alpha-\lambda-\gamma|\gamma+\lambda,[\varepsilon,T]}\notag\\
\le & C_{\alpha,\eta,\gamma,\alpha-\lambda}\|x\|_\eta(\|\sigma(y)\|_{\alpha,[\varepsilon,T]}
+\|\hat\delta_1\sigma(y)\|_{\alpha-\lambda|\lambda,[\varepsilon,T]})\notag\\
\le & c_5\varepsilon^{\lambda-\alpha}
\label{stima-utile-3}
\end{align}
for some positive constant $c_5=c_5(\alpha,\eta,\sigma,\|x\|_{C^{\eta}([0,T])},\mathfrak R,T,\lambda,\gamma,|\psi|_{\alpha})$, which does not depend on $\varepsilon$.
From \eqref{scomp_soluzione}, with $\tau=\varepsilon$, we can write
\begin{equation}
y(t)=S(t-\varepsilon)(y(\varepsilon))+{\mathscr I}_{S\sigma(y)}(\varepsilon,t),\qquad\;\,t\in[\varepsilon,T]
\label{aaaa}
\end{equation}
and applying \eqref{reg_sol}, with $t=\varepsilon$ and $r=\alpha$, \eqref{stima-utile-3} and \eqref{aaaa} we infer that
\begin{align}
|y(t)|_{\gamma+\lambda}
\leq & M_{\alpha,\gamma+\lambda,T}(t-\varepsilon)^{\alpha-{\gamma-\lambda}}|y(\varepsilon)|_\alpha+|({\mathscr I}_{S\sigma(y)}(\varepsilon,t)|_{\gamma+\lambda}\notag  \\
 \leq & c_1M_{\alpha,\gamma+\lambda,T}(t-\varepsilon)^{\alpha-{\gamma-\lambda}}\notag  \\
 & +\|{\mathscr I}_{S\sigma(y)}\|_{\eta+\alpha-\lambda-\gamma|\lambda+\gamma,[\varepsilon,T]}(t-\varepsilon)^{\eta+\alpha-\lambda-\gamma} \notag \\
\leq & c_6(t-\varepsilon)^{\alpha-\gamma-\lambda}\varepsilon^{\lambda-\alpha},
\label{stima_mild_solution_reg}
\end{align}
for every $t\in(\varepsilon,T]$ and some positive constant $c_6=c_6(\lambda,\gamma,\eta,\alpha,\sigma,x,\psi,\mathfrak R, T)$. In particular,
since the range of the function $\varrho:D\to\R$, defined by $\varrho(\lambda,\gamma)=\lambda+\gamma$ for every $(\lambda,\gamma)\in D=\{
(\lambda,\gamma)\in\R^2: \lambda\in [0,\eta+\alpha-1), \gamma\in [\alpha-\lambda,1)\}$ is the interval $[\eta+\alpha-1,\eta+\alpha)$,
for every $\mu\in[0,\eta+\alpha-1)$ we can choose $\lambda$ and $\gamma$ such that $1+\mu=\lambda+\gamma$. Then, from \eqref{stima_mild_solution_reg} we conclude that
\begin{align*}
|y(t)|_{1+\mu}
\leq c_7(t-\varepsilon)^{\alpha-1-\mu}\varepsilon^{\lambda-\alpha}, \qquad\;\, t\in(\varepsilon,T],
\end{align*}
so that, for every $\varepsilon\in (0,T/2)$,
\begin{align}
\label{stima_epsilon_2}
|y(t)|_{1+\mu}\leq c_7\varepsilon^{\lambda-1-\mu}, \qquad\;\, t\in [2\varepsilon,T],
\end{align}
and $c_7=c_7(\lambda,\mu,\eta,\alpha,\sigma,x,\psi,\mathfrak R, T)$ is a positive constant, which depends in a continuous way on $\|x\|_{C^{\eta}([0,T])}$ and on $\mathfrak R$ but not on $\varepsilon$.
From \eqref{stima_epsilon_2}, estimate \eqref{stima_lambda_mu} follows at once. Finally, using again \eqref{aaaa} and the smoothness properties of the semigroup $(S(t))_{t\ge 0}$, we conclude that $y\in \hat C_{\eta+\alpha-\mu}([2\varepsilon,T];X_{\mu})$ for every $\mu\in [\eta+\alpha-1,\eta+\alpha)$ and $\varepsilon\in (0,T/2)$.
\end{proof}

\begin{remark}
\label{rmk:remark_teorema_principale}
{\rm \begin{enumerate}[\rm (i)]
\item
Theorem \ref{es_mild_sol_gub} generalizes the results in \cite[Theorem 4.3]{GT10}.
\item
From the last part of Step $3$ in the proof of Theorem \ref{es_mild_sol_gub} it follows that $y\in C((0,T];X_\mu)$ for any $\mu\in[0,\eta+\alpha)$.
\item
In Step 3 of the proof of Theorem \ref{es_mild_sol_gub} we have proved that for each $r\in[\alpha,1)$ there exists a constant $c$ such that
\begin{align}
\label{stima_sol_Xr}
|y(t)|_{r}\leq ct^{\alpha-r}, \qquad\;\, t\in(0,T],
\end{align}
for some constant $c$, independent of $t$.
If $\psi\in X_{\gamma}$ for some $\gamma\in [\alpha,1)$, then arguing as in estimate \eqref{reg_sol}, we can easily show that we can replace
$\alpha-r$ with $(\gamma-r)\wedge 0$ in \eqref{stima_sol_Xr}, with $r\in[\alpha,1)$. Based on this estimate, \eqref{questo} and \eqref{quello}, we conclude that
\begin{align}
|y(t)-y(s)|\le &|(\hat\delta_1y)(t,s)|_r+|\mathfrak{a}(s,y)y(s)|_r\notag\\
\le &c_{*}(t-s)^{\eta+\gamma-r}+c_{**}s^{(\gamma-r-\beta)\wedge 0}|t-s|^{\beta}
\label{psico}
\end{align}
for every $\beta>0$ such that $r+\beta<1$, every $0<s<t\le T$ and some positive constants $c_*$ and $c_{**}$, independent of $s$ and $t$. Since $\beta<\eta+\gamma-r$, from \eqref{psico} we conclude that
\begin{align*}
|y(t)-y(s)|_r\leq c s^{(\gamma-r-\beta)\wedge 0}|t-s|^{\beta}, \qquad\;\, 0<s<t\leq T.
\end{align*}
If $\gamma-r-\beta\ge 0$ then the above estimate can be extended to $s=0$.
We will use these estimates in Section \ref{sect-examples}.
\end{enumerate}
}
\end{remark}

\begin{remark}
\label{rmk:mild_sol_vett}
{\rm The result in Theorem \ref{es_mild_sol_gub} extend, using the same techniques, to the case of the Young equation
\begin{align}
\label{equazione_vettoriale}
\left\{
\begin{array}{ll}
dy(t)=Ay(t)dt+\displaystyle\sum_{i=1}^m\sigma_i(y(t))dx_i(t), &t\in(0,T],\\[1mm]
y(0)=\psi.
\end{array}
\right.
\end{align}
where the nonlinear terms $\sigma_i$ $(i=1,\ldots,m)$ satisfy Hypotheses \eqref{ip_nonlineare} and the paths $x_i\in C^\eta([0,T])$ $(i=1,\ldots,n)$ belong to
$C^{\eta}([0,T])$.}
\end{remark}

\section{The integral solution}
\label{sect-4}
To deal with integral solutions, we prove that the integral
\begin{align*}
\int_0^t\sigma(y(u))dx(u), \qquad\;\, 0\leq t\leq T,
\end{align*}
where $y$ is the unique mild solution to \eqref{omo_equation}, is well defined as Young integral, see Theorem \ref{thm:young_int}.

\begin{definition}
Let $y\in \hat C_\alpha([0,T];X_\alpha)\cap L^1((0,T);D(A))$ for some $\alpha\in (1-\eta,1)$. We say that $y$ is an integral solution to \eqref{omo_equation} if it satisfies the integral equation
\begin{align}
\label{int_solution}
y(t)=\psi+\int_0^tAy(u)du+\int_0^t\sigma(y(u))dx(u), \qquad\;\, t\in[0,T].
\end{align}
\end{definition}

\begin{remark}
{\rm If $\sigma$ satisfies Hypothesis $\ref{ip_nonlineare}$, then for every $f\in \hat C_\alpha([0,T];X_\alpha)$ and $x\in C^\eta([0,T])$, where $\eta\in (1/2,1)$ and $\alpha\in (1-\eta,1)$}, the Young integral
\begin{align}
\int_s^t\sigma(f(u))dx(u), \qquad\;\,(s,t)\in [0,T]^2_<,
\label{scuola}
\end{align}
is well defined. Indeed, arguing as in the proof of \eqref{stima_delta} it can be easily checked that
$\sigma(f)\in C^{\alpha}([0,T];X)$. Therefore, Theorem $\ref{thm:young_int}$ guarantees that the integral in \eqref{scuola} is well defined.
\end{remark}

The main result of this section shows that, under Hypotheses \ref{hyp-main}, \ref{ip:young_path} and \ref{ip_nonlineare}, the mild solution $y$ to \eqref{omo_equation} is also an integral solution.
To prove such a result, we first prove that the mild solution to \eqref{omo_equation} can be approximated by mild solutions of classical problems.

\begin{proposition}
\label{conv_sol_appr}
Let $(x_n)\subset C^1([0,T])$ be a sequence converging to $x$ in $C^{\eta}([0,T])$ for some $\eta>1/2$ and fix $\psi\in X_{\alpha}$ for some $\alpha\in (0,1/2)$ such that $\alpha+\eta>1$. For every $n\in\N$, denote by $y_n$ the mild solution to \eqref{omo_equation} with $x$ replaced by $x_n$, and let $y$ be the mild solution to \eqref{omo_equation}. Then, the following properties are satisfied:
\begin{enumerate}[\rm (i)]
\item
$y_n$ converges to $y$ in $\hat C^{\alpha}([0,T];X_{\alpha})$ as $n$ tends to $+\infty$;
\item
if we set $\displaystyle\mathbb J(t)=\int_0^t \sigma(y(u))dx(u)$ and $\mathbb J_n(t)=\displaystyle\int_0^t \sigma(y_n(u))dx_n(u)$ for every $t\in[0,T]$ and $n\in\N$,
then $\mathbb J_n$ converges to $\mathbb J$ in $C^{\eta}([0,T];X)$ as $n$ tends to $+\infty$.
\end{enumerate}
\end{proposition}

\begin{proof}
(i) We split the proof into two steps. In the first one, we show the assertion when $T$ is small enough and in the second step we remove this additional condition.

{\em Step 1}. Let us fix $\tau,\widetilde T\in[0,T]$ with $\tau<\widetilde T$. To begin with we observe that, applying Lemma \ref{thm:young_integral}, with $k=r=\beta_1=\alpha$, $\beta=0$ and $a=\tau,b=\widetilde T$ and noticing that, by Corollary \ref{integrale_ancorato} with $a=\tau$ and $b=\widetilde T$,
$(\hat\delta_1{\mathscr I}_{Sf}(\tau,\cdot))(s,t)={\mathscr I}_{Sf}(s,t)$ for every $(s,t)\in [\tau,\widetilde T]$, we can show that
\begin{align}
\|{\mathscr I}_{Sf}(\tau,\cdot)\|_{\alpha,\alpha,[\tau,\widetilde T]}
\leq C(\widetilde T-\tau)^{\eta-\alpha} (\|f\|_{\alpha, [\tau,\widetilde T]}+\|\hat\delta_1f\|_{\alpha|0,[\tau,\widetilde T]})\|x\|_{C^{\eta}([0,T])},
\label{suzanne}
\end{align}
for every $x\in C^{\eta}([0,T])$ and $f\in \hat C_{\alpha}([0,T];X)\cap C([0,T];X_{\alpha})$ such that $\alpha\in (0,1/2)$ and $\eta+\alpha>1$, and $\tau, \widetilde T\in [0,T]$, with $\tau<\widetilde T$, where $C=C_{\alpha,\eta,\alpha,\alpha}$ is the constant in Lemma \ref{thm:young_integral}.

Now, we fix ${T_*}\in(0,T]$ to be chosen later on.  From \eqref{young_integral_linearity} and \eqref{young_integral_linearity_1} we get
\begin{align*}
y(t)-y_n(t)
= & \int_0^tS(t-r)\sigma(y(r))dx(r)-\int_0^tS(t-r)\sigma(y_n(r))dx_n(r) \notag \\
= & \int_0^tS(t-r)\sigma(y_n(r))d\overline x_n(r)+\int_0^tS(t-r)(\sigma(y(r))-\sigma(y_n(r)))dx(r)\notag\\
=&\!: \mathbb I_{1,n}(t)+\mathbb I_{2,n}(t)
\end{align*}
for every $t\in[0,{T_*}]$,  where $\overline x_n:=x-x_n$. Taking \eqref{stima_sigma_1} and \eqref{stima_sigma_2} into account, we can estimate
\begin{align}
\|\mathbb I_{1,n}\|_{\alpha,\alpha,[0,{T_*}]}\leq CT_*^{\eta-\alpha}(\|\sigma(y_n)\|_{\alpha,[0,{T_*}]}+\|\hat\delta_1 \sigma(y_n)\|_{\alpha|0,[0,{T_*}]})\|\overline x_n\|_{C^{\eta}([0,T])}\notag\\
\le  CT_*^{\eta-\alpha}(L_\sigma+L_\sigma^\alpha)(C_{\alpha,0,T}+2)(1+\|y_n\|_{\alpha,\alpha,[0,{T_*}]})\|\overline x_n\|_{C^{\eta}([0,T])},
\label{parete}
\end{align}
where $C$ is a positive constant which depends on $\alpha$, $\eta$, $\sigma$ and $T$. An inspection of the proof of estimate \eqref{aprioir_est_sol} shows that the constant $\mathfrak{R}$ depends in a continuous way on the $\eta$-H\"older norm of the path.
Since $\sup_{n\in\N}\|x_n\|_{C^{\eta}([0,T])}<+\infty$,  from \eqref{parete} we can infer that
\begin{align*}
\|\mathbb I_{1,n}\|_{\alpha,\alpha,[0,{T_*}]}\leq cT_*^{\eta-\alpha}(L_\sigma+L_\sigma^\alpha)(C_{\alpha,0,T}+2)(1+\mathfrak M)\|\overline x_n\|_{C^{\eta}([0,T])},
\end{align*}
for some positive constant $\mathfrak M$, independent of $n$. As far as ${\mathbb I}_{2,n}$ is considered, from \eqref{suzanne}, with $f$ replaced by $\sigma(y)-\sigma(y_n)$, and estimates
\eqref{approx_stima_2_1}, \eqref{approx_stima_2_2}, we infer that
\begin{align*}
&\|\mathbb I_{2,n}\|_{\alpha,\alpha,[0,{T_*}]}\notag\\
\leq &cT_*^{\eta-\alpha}
(\|\sigma(y)-\sigma(y_n)\|_{\alpha,[0,{T_*}]}
+\|\hat\delta_1(\sigma(y)- \sigma(y_n))\|_{\alpha|0,[0,{T_*}]})\|x\|_{C^{\eta}([0,T])}\notag\\
\le & \widetilde cT_*^{\eta-\alpha}\|y-y_n\|_{\alpha,\alpha,[0,{T_*}]}\|x\|_{C^{\eta}([0,T])},
\end{align*}
and $\widetilde c$ is a positive constant which depends on $\alpha$, $T$, $\sigma$, $\mathfrak M$, $K$, $\eta$ and on the constant $C_{\alpha,0,T}$. We choose ${T_*}\leq T$ such that $\widetilde cT_*^{\eta-\alpha}\|x\|_{C^{\eta}([0,T])}\leq 1/2$ and use the previous estimate to conclude that
\begin{align*}
\|y-y_n\|_{\alpha,\alpha,[0,{T_*}]}
\leq 2 cT_*^{\eta-\alpha}(L_\sigma+L_\sigma^\alpha)(C_{\alpha,0,T}+2)(1+\mathfrak M)\|x_n-x\|_{C^{\eta}([0,T])}
\end{align*}
and, consequently, that $y_n$ converges to $y$ in $\hat C_{\alpha}([0,T_*];X_{\alpha})$ as $n$ tends to $+\infty$.

{\em Step 2}. If $T_*= T$ then we are done. Otherwise, let us fix $\widehat T:=(2T_*)\wedge T$. For every $t\in[ T_*,\widehat T]$, from \eqref{scomp_soluzione} we can write
\begin{align*}
y(t)-y_n(t)
= & \int_{{T_*}}^tS(t-r)\left(\sigma(y(r))-\sigma(y_n(r))\right)dx(r)
+\int_{{T_*}}^tS(t-r)\sigma(y_n(r))d\overline x_n (r) \\
& + S(t-{T_*})(y({T_*})-y_n({T_*})).
\end{align*}
In Step $1$ we have proved that $y_n({T_*})$ converges to $y({T_*})$ in $X_\alpha$ as $n$ tends to $+\infty$. Moreover, for every $(s,t)\in [T_*,T]^2_{<}$ it holds that
$\widehat\delta_1S(\cdot-{T_*})(y({T_*})-y_n({T_*}))(s,t)=0$.
Hence,
$\|S(\cdot-{T_*})(y({T_*})-y_n({T_*}))\|_{\alpha,\alpha,[{T_*},\widehat T]}$ vanishes as $n$ tends to $+\infty$.
Repeating the same arguments as in Step 1, we conclude that
\begin{align*}
\|y-y_n\|_{\alpha,\alpha,[{T_*},\widehat T]}
\leq 2cT_*^{\eta-\alpha}(L_\sigma+L_\sigma^\alpha)(C_{\alpha,0,T}+2)(1+\mathfrak M)\|x_n-x\|_{C^{\eta}([0,T])},
\end{align*}
and therefore $y_n$ converges to $y$ in $\hat C_{\alpha}([{T_*},\widehat T];X_{\alpha})$ as $n$ tends to $+\infty$.
If $\widehat T=T$ then the assertion follows. Otherwise by iterating this argument, we get the assertion in a finite number of steps.
\medskip

(ii) As in the proof of property (i),  we can write
\begin{align*}
\mathbb J_n(t)-\mathbb J(t)
= & \int_0^t\sigma(y_n(u))d\overline x_n(u)+\int_0^t(\sigma(y_n(u))-\sigma(y(u)))dx(u) \\
=&\!:\mathbb J_1^n(0,t)+\mathbb J_2^n(0,t).
\end{align*}
From \eqref{stima_int_young_nosemi}, \eqref{stima_sigma_1} and \eqref{stima_sigma_2}, we infer
\begin{align}
|\mathbb J_1^n(0,t)|_0
\leq & t^{\eta}\bigg (\|\sigma(y_n)\|_{0,[0,T]}+\frac{\|\sigma(y_n)\|_{\alpha,\alpha,[0,T]}}{1-2^{-(\eta+\alpha-1)}}\bigg )\|\overline x_n\|_{C^{\eta}([0,T])}\notag \\
\leq & T^{\eta}\bigg (L_{\sigma}+\frac{(L_\sigma+L_{\sigma}^{\alpha})(1+C_{\alpha,0,T})}{1-2^{-(\eta+\alpha-1)}}\bigg )\bigg (1+\sup_{n\in\N}\|y_n\|_{\alpha,\alpha,[0,T]}\bigg )\|\overline x_n\|_{C^{\eta}([0,T])}
\label{stima_Jsigma1}
\end{align}
for every $t\in[0,T]$,
As far as the term $\mathbb J_2^n(0,t)$ is concerned, we argue similarly, taking advantage
of the computations in \eqref{A} and estimate \eqref{approx_stima_2_1}, and get
\begin{align}
|\mathbb J_2^n(0,t)|\le CT^{\eta}\|y-y_n\|_{\alpha|\alpha,[0,T]}.
\label{stima_Jsigma2}
\end{align}
From \eqref{stima_Jsigma1} and \eqref{stima_Jsigma2} it thus follows that
\begin{eqnarray*}
\sup_{t\in [0,T]}|\mathbb J_n(t)-\mathbb J(t)|\le C'T^{\eta}(\|x-x_n\|_{C^{\eta}([0,T])}+\|y-y_n\|_{\alpha|\alpha,[0,T]})
\end{eqnarray*}
for a suitable constant $C'$, independent of $n$.
From the assumptions on $x$ and $(x_n)$, and property (i), we conclude that $\mathbb J_n$ converges to $\mathbb J$ in $C([0,T];X)$ as $n$ tends to $+\infty$.

To prove that ${\mathbb J}_n$ converges to  ${\mathbb J}$ in $C^{\eta}([0,T];X)$, now it suffices to note that (see Remark \ref{rem-davide})
\begin{align*}
(\delta_1(\mathbb J_n-\mathbb J))(s,t)={\mathbb J}_1^n(s,t)+{\mathbb J}_2^n(s,t), \qquad\;\,(s,t)\in [0,T]^2_{<}.
\end{align*}
and repeat the above computations to infer that
\begin{eqnarray*}
[\mathbb J_n-\mathbb J]_{C^{\eta}([0,T])}\le C'(\|x-x_n\|_{C^{\eta}([0,T])}+\|y-y_n\|_{\alpha|\alpha,[0,T]})
\end{eqnarray*}
for every $n\in\N$.
\end{proof}

We are now ready to show that the mild solution $y$ to \eqref{omo_equation} is an integral solution.
\begin{theorem}\label{th-soluzione-classica}
Let Hypotheses $\ref{hyp-main}$, $\ref{ip:young_path}$ and $\ref{ip_nonlineare}$ be satisfied and let $\psi\in X_{\alpha}$ for some $\alpha\in (0,1)$ such that $\alpha+\eta>1$.
Further, let $y$ be the unique mild solution to equation \eqref{omo_equation}. Then, $y$ satisfies \eqref{int_solution}.
\end{theorem}

\begin{proof}
Let $(x_n)\subset C^1([0,T])$ be a sequence of smooth paths which converges to $x$ in $C^{\eta}([0,T])$ as $n$ tends to $+\infty$. For every $n\in\N$, let $y_n$ be the unique mild solution to \eqref{omo_equation} with $x$ replaced by $x_n$.
The computations in Step 3 of the proof of Theorem \ref{es_mild_sol_gub} with $x$ replaced by $x_n$, the fact that $\sup_{n\in\N}\|x_n\|_{\eta}<+\infty$ imply that $y_n(t)$ belongs to $D(A)$ for each $t\in(0,T]$ and $n\in\N$, and for every $\lambda\in [0,\eta+\alpha-1)$ there exists a positive constant $c=c(\lambda)$, independent of $n$, such that
$|Ay(t)|_0\leq ct^{\lambda-1}$ and $|Ay_n(t)|_0\leq ct^{\lambda-1}$ for every $t\in(0,T]$ and $n\in\N$.
From \cite[Proposition 4.1.5]{LU95} we infer that
\begin{align*}
y_n(t)=\psi+\int_0^t A y_n(s)ds+\int_0^t\sigma(y_n(s))dx_n(s), \qquad t\in[0,T], \;\, n\in\N.
\end{align*}
Let us fix $t\in(0,T]$. From Proposition \ref{conv_sol_appr} we know that $y_n$ converges to $y$ in $C([0,T];X)$
and $\displaystyle\int_0^t\sigma(y_n(s))dx_n(s)$ converges to $\displaystyle\int_0^t\sigma(y(s))dx(s)$ in $X$ as $n$ tends to $+\infty$.
Hence,
$\displaystyle\int_0^ty_n(s)ds$ and
\begin{align*}
A\int_0^ty_n(s)ds
= & \int_0^tAy_n(s)ds=y_n(t)-\psi-\int_0^t\sigma(y_n(s))dx_n(s)
\end{align*}
converge, as $n$ tends to $+\infty$, to $\displaystyle\int_0^ty(s)ds$ and
$\displaystyle y(t)-\psi-\int_0^t\sigma(y(s))dx(s)$,
respectively, for every $t\in [0,T]$.
Since $A$ is a closed operator it follows that
\begin{align*}
\int_0^ty(s)ds\in D(A),\qquad\;\, A\int_0^ty(s)ds=y(t)-\psi-\int_0^t\sigma(y(s))dx(s).
\end{align*}
Finally, since $|Ay(t)|\le ct^{\lambda-1}$ for every $t\in (0,T]$ (see \eqref{stima_lambda_mu} with $\mu=0$), it follows that $Ay$ belongs to $L^1(0,T;X)$. Hence,
$\displaystyle A\int_0^ty(s)ds=\int_0^tAy(s)ds$,
which gives
\begin{align*}
y(t)=\psi+\int_0^tAy(s)ds+\int_0^t\sigma(y(s))dx(s).
\end{align*}
The arbitrariness of $t\in[0,T]$ yields the assertion.
\end{proof}

\begin{corollary}
\label{coro:int_sol_vett}
Let $\sigma_i:X\rightarrow X$ $(i=1,\ldots,m)$ satisfy Hypotheses \eqref{ip_nonlineare} and let the paths $x_i\in C^\eta([0,T])$ $(i=1,\ldots,n)$ belong to $C^{\eta}([0,T])$. Then, the unique mild solution $y$ to \eqref{equazione_vettoriale} with $\psi\in X_\alpha$, with $\alpha+\eta>1$, satisfy
\begin{align}
\label{int_solution_vett}
y(t)=\psi+\int_0^tAy(u)du+\sum_{i=1}^m\int_0^t\sigma_i(y(u))dx_i(u), \qquad\;\, t\in[0,T].
\end{align}
\end{corollary}
\begin{proof}
The statement follows from Remark \ref{rmk:mild_sol_vett}, and by repeating the computation in this section.
\end{proof}

\section{Chain rule for non-linear Young equations}
\label{sect-5}
In this subsection we use the integral representation \eqref{int_solution} of the unique mild solution $y$ to problem \eqref{omo_equation} to prove a chain rule for $F(\cdot,y(\cdot))$, where $F$ is a smooth function.
\begin{theorem}
\label{thm:chain_rula}
Let $F\in C^1([0,T]\times X)$ be such that and $F_x$ is $\alpha$-H\"older continuous with respect to $t$, locally uniformly with respect to $x$ and is locally $\gamma$-H\"older continuous with respect to $x$,
uniformly with respect to $t$, for some $\alpha, \gamma\in (0,1)$ such that $\eta +\alpha\gamma>1$. Further,  let $y$ be the unique mild solution to \eqref{omo_equation}. Then,
\begin{align*}
F(t,y(t))-F(s,y(s))
=&  \int_s^tF_t(u,y(u))du+\int_s^t\langle F_x(u,y(u)),Ay(u)\rangle du\notag\\
& +\int_s^t\langle F_x(u,y(u)),\sigma(y(u))\rangle dx(u)
\end{align*}
for every $(s,t)\in [0,T]$.
\end{theorem}

\begin{proof}
Let us fix $0<s<t\leq T$ and a sequence $(\Pi_n(s,t))$ of partitions $\Pi_n(s,t)=\{s=s_0^n<s_1^n<\ldots<s_{m_n}^n=t\}$ of $[s,t]$ and note that
\begin{align*}
&F(t,y(t))-F(s,y(s))\\
=  & \sum_{j=1}^{m_n}F(s^n_j,y(s^n_j)) -F(s^n_{j-1},y(s^n_{j-1})) \\
= & \sum_{j=1}^{m_n}[F(s^n_j,y(s^n_j))-F(s^n_{j-1},y(s^n_{j}))
+F(s^n_{j-1},y(s^n_j))-F(s^n_{j-1},y(s^n_{j-1}))]\\
= &\sum_{j=1}^{m_n}F_t(s^n_{j},y(s^n_j))\Delta s^n_j
+\sum_{j=1}^{m_n}\big(F_t(\tilde s^n_{j},y(s^n_{j}))-F_t(s^n_{j},y(s^n_{j}))\big)\Delta s^n_j  \\
& + \sum_{j=1}^{m_n}\langle F_x(s^n_{j-1},y(s^n_{j-1})),\Delta y_j\rangle
+\sum_{j=1}^{m_n} \langle F_{x}(s^n_{j-1},\tilde y_j)-F_x(s^n_{j-1},y(s^n_{j-1})),\Delta y_j\rangle \\
=: & I_{1,n}+I_{2,n}+I_{3,n}+I_{4,n},
\end{align*}
where
$\Delta y_j=y(s^n_j)-y(s^n_{j-1})$, $\Delta s^n_j=s^n_j-s^n_{j-1}$, $\tilde s^n_j=s^n_{j-1}+\theta_j^n(s^n_j-s^n_{j-1})$, $\tilde y_j=y(s^n_{j-1})+\eta_j^n(y(s^n_j)-y(s^n_{j-1}))$ and
$\theta_j^n,\eta_j^n\in (0,1)$ for every $j=1,\ldots,m_n$.
Without loss of generality we can assume that $|\Pi_n(s,t)|$ tends to zero as $n$ tends to $+\infty$.

{\em Analysis of the terms $I_{1,n}$ and $I_{2,n}$}. Since the function $s\mapsto F_t(s,y(s))$ is continuous in $[0,T]$, $I_{1,n}$ converges to $\displaystyle\int_s^tF_t(u,y(u))du$ as $n$ tends to $+\infty$.
Moreover, since $y([0,T])$ is compact subset of $X$, the restriction of function $F_t$ to $[0,T]\times y([0,T])$ is uniformly continuous.
Thus, for every $\varepsilon>0$ there exists a positive constant $\delta$ such that $|F_t(t_2,x_2)-F_t(t_1,x_1)|\le\varepsilon$ if $|t_2-t_1|^2+|x_2-x_1|^2\le\delta^2$. As a byproduct, it follows that, if $|\Pi(s,t)|\le\delta$, then
$|I_{2,n}|\le\varepsilon\sum_{j=1}^n\Delta s^n_j=\varepsilon(t-s)$ and this shows that $I_{2,n}$ converges to $0$ as $n$ tends to $+\infty$.

{\em Analysis of the term $I_{3,n}$}. Using \eqref{int_solution} we can write (see Remark \ref{rem-davide})
\begin{align}
&\langle F_x(s^n_{j-1},y(s^n_{j-1})),\Delta y_j\rangle \notag \\
=& \bigg\langle F_x(s^n_{j-1},y(s^n_{j-1})),\int_{s^n_{j-1}}^{s^n_j}Ay(u)du+\int_{s^n_{j-1}}^{s^n_j}\sigma(y(u))dx(u)\bigg\rangle  \notag \\
= & \langle F_x(s^n_{j-1},y(s^n_{j-1})),Ay(s^n_{j-1})\rangle \Delta s^n_j \notag \\
&+\bigg\langle F_x(s^n_{j-1},y(s^n_{j-1})),
\int_{s^n_{j-1}}^{s^n_j}(Ay(u)-Ay(s^n_{j-1}))du\bigg\rangle \notag \\
& +\bigg\langle F_x(s^n_{j-1},y(s^n_{j-1})), \int_{s^n_{j-1}}^{s^n_j}(\sigma(y(u))-\sigma(y(s^n_{j-1})))dx(u)\bigg\rangle\notag\\
& +\langle F_x(s^n_{j-1},y(s^n_{j-1})),\sigma(y(s^n_{j-1}))\rangle (x(s^n_j)-x(s^n_{j-1}))
\label{divisione_I_3}
\end{align}
for $j=1,\ldots,m$. By assumptions, the function $s\mapsto F_x(s,y(s))$ is continuous with values in $X'$. Similarly, by Theorem \ref{es_mild_sol_gub} the function $Ay$ is continuous in $(0,T]$. Indeed, $y$ belongs to $\hat C_{\eta+\alpha-\mu}([s,T];X_{\mu})$ for every $\mu\in [\eta+\alpha-1,\eta+\alpha)$. Taking $\mu=1$ we deduce that
$|(\hat\delta_1y)(u,w)|_1\leq c|w-u|^{\eta+\alpha-1}$ for every $(u,w)\in[s,t]_<^2$ and some positive constant $c$, independent of $u$ and $w$. Hence,
\begin{align*}
|Ay(u)-Ay(w)|
\le & |(\hat \delta_1y)(u,w)|_1 +|{\mathfrak a}(u,w)Ay(u)|_0 \\
\le & c|w-u|^{\eta+\alpha-1}+|{\mathfrak a}(u,w)Ay(u)|_0.
\end{align*}
Choosing $\mu=1+\rho$ for some $\rho<\eta+\alpha-1$ and using \eqref{stime_smgr}$(b)$ we get
\begin{align*}
|a(u,w)Ay(w)|_0\leq C_{0,\rho}|w-u|^{\rho}|Ay(w)|_{\rho}
\leq C_{0,\rho}|w-u|^{\rho}\|y\|_{1+\rho,[\varepsilon,T]}.
\end{align*}
Therefore, $Ay$ is $\rho$-H\"older continuous in $[\varepsilon,T]$. We thus conclude that
\begin{align}
\label{convergenza_1_chain_rule}
\lim_{n\to+\infty}\sum_{j=1}^{m_n}\langle F_x(s^n_{j-1},y(s^n_{j-1})),Ay(s^n_{j-1})\rangle \Delta s^n_j=\int_s^t\langle F_x(u,y(u)),Ay(u)\rangle du
\end{align}
and
\begin{align*}
\bigg|\int_{s^n_{j-1}}^{s^n_j}(Ay(u)-Ay(s^n_{j-1}))du\bigg |\le & [Ay]_{C^{\rho}([s,T];X)}|s^n_j-s^n_{j-1}|^{1+\rho}\\
\le &[Ay]_{C^{\rho}([s,T];X)}(s^n_j-s^n_{j-1})|\Pi(s,t)|^{\rho}
\end{align*}
for every $j=1,\ldots,m$, so that
\begin{align}
&\bigg |\sum_{j=1}^{m_n}\bigg\langle F_x(s^n_{j-1},y(s^n_{j-1})),\int_{s^n_{j-1}}^{s^n_j}(Ay(u)-Ay(s^n_{j-1}))du\bigg \rangle\bigg |\notag\\
\le &\|F_x\|_{C([0,T]\times y([0,T]);X')}[Ay]_{C^{\rho}([s,T];X)}(t-s)|\Pi_n(s,t)|^{\rho}
\label{convergenza_2_chain_rule}
\end{align}
and the right-hand side of the previous inequality vanishes as $n$ tends to $+\infty$.

Let us consider the third term in the right-hand side of \eqref{divisione_I_3}.
From Theorem \ref{thm:young_int} and recalling that $\alpha+\eta>1$, we infer that
\begin{align*}
&\left|\int_{s^n_{j-1}}^{s^n_j}(\sigma(y(u))-\sigma(y(s^n_{j-1})))dx(u)\right|_0 \\
= & \left|\int_{s^n_{j-1}}^{s^n_j}\sigma(y(u))dx(u)-\sigma(y(s^n_{j-1}))(x(s^n_j)-x(s^n_{j-1})) \right|_0\\
\leq  & \frac{1}{1-2^{\alpha-\eta-1}}\|\delta_1\sigma(y)\|_{\alpha|0,[0,T]}\|x\|_{C^{\eta}([0,T])}|s^n_j-s^n_{j-1}|^{\alpha+\eta}\\
\leq  & \frac{1}{1-2^{\alpha-\eta-1}}\|\delta_1\sigma(y)\|_{\alpha|0,[0,T]}\|x\|_{C^{\eta}([0,T])}(s^n_j-s^n_{j-1})|\Pi_n(s,t)|^{\alpha+\eta-1}\\
\end{align*}
for $j=1,\ldots,m$. Hence,
\begin{align*}
&\bigg |\sum_{j=1}^{m_n}\bigg\langle F_x(s^n_{j-1},y(s^n_{j-1})),\int_{s^n_{j-1}}^{s^n_j}(\sigma(y(u))-\sigma(y(s^n_{j-1})))dx(u)\bigg\rangle \bigg |\\
\le & \frac{1}{1-2^{\alpha-\eta-1}}\|F_x\|_{C([0,T]\times y([0,T]);X')}(t-s)|\Pi_n(s,t)|^{\alpha+\eta-1}.
\end{align*}
Letting $n$ tend to $+\infty$ gives
\begin{align}
\label{convergenza_3_chain_rule}
\lim_{n\to +\infty}\sum_{j=1}^{m_n}\bigg\langle F_x(s^n_{j-1},y(s^n_{j-1})),\int_{s^n_{j-1}}^{s^n_j}(\sigma(y(u))-\sigma(y(s^n_{j-1})))dx(u)\bigg\rangle  =0.
\end{align}

To conclude the study of $I_{3,n}$ it remains to consider the term
\begin{align*}
\langle F_x(s^n_{j-1},y(s^n_{j-1})),\sigma(y(s^n_{j-1}))\rangle (x(s^n_j)-x(s^n_{j-1})).
\end{align*}
For this purpose, we introduce the function $g:[s,t]\to\R$, defined by $g(\tau)=\langle F_x(\tau,y(\tau)),\sigma(y(\tau))\rangle$ for every $\tau\in [s,t]$. Let us prove that
$g\in C^{\alpha\gamma}([s,t])$. To this aim, we recall that
\begin{eqnarray*}
|\sigma(y(\tau))|_0\le K_{0,\alpha}|\sigma(y(\tau))|_{\alpha}\le K_{0,\alpha}L_{\sigma}^{\alpha}(1+\|y\|_{\alpha,[0,T]}),\qquad\;\,\tau\in [0,T].
\end{eqnarray*}
Hence, we can estimate
\begin{align*}
|g(\tau_2)-g(\tau_1)|=&|\langle F_x(\tau_2,y(\tau_2)),\sigma(y(\tau_2))\rangle-\langle F_x(\tau_1
,y(\tau_1)),\sigma(y(\tau_1))\rangle|\\
\le &|\langle F_x(\tau_2,y(\tau_2))-F_x(\tau_2,y(\tau_1)),\sigma(y(\tau_2))\rangle|\\
&+\langle F_x(\tau_2,y(\tau_1))-F_x(\tau_1,y(\tau_1)),\sigma(y(\tau_2))\rangle\\
&+|\langle F_x(\tau_1,y(\tau_1)),\sigma(y(\tau_2))-\sigma(y(\tau_1))\rangle|\\
\le & K_{0,\alpha}L_{\sigma}^{\alpha}\sup_{t\in[0,T]}\|F_x(t,\cdot)\|_{C^{\gamma}(y([0,T]);X')}(1+\|y\|_{\alpha,[0,T]})|y(\tau_2)-y(\tau_1)|_0^{\gamma}\\
&+K_{0,\alpha}L_{\sigma}^{\alpha}\sup_{x\in y([0,T])}[F_x(\cdot,x)]_{C^{\alpha}([0,T];X')}(1+\|y\|_{\alpha,[0,T]})|\tau_2-\tau_1|^{\alpha}\\
&+L_{\sigma}\|F_x\|_{C([0,T]\times y([0,T]);X')}|y(\tau_2)-y(\tau_1)|_0 \\
\le &\bigg (K_{0,\alpha}L_{\sigma}^{\alpha}\sup_{t\in[0,T]}\|F_x(t,\cdot)\|_{C^\gamma(y([0,T]);X')}(1+\|y\|_{\alpha,[0,T]})[y]_{C^{\alpha}([0,T];X)}\\
&\;\;\,+K_{0,\alpha}L_{\sigma}^{\alpha}\sup_{x\in y([0,T])}[F_x(\cdot,x)]_{C^{\alpha}([0,T];X')}(1+\|y\|_{\alpha,[0,T]})T^{\alpha(1-\gamma)}\\
&\;\;\,+L_{\sigma}\|F_x\|_{C([0,T]\times y([0,T]);X')}[y]_{C^{\alpha}([0,T];X)}T^{\alpha(1-\gamma)}\bigg )|\tau_2-\tau_1|^{\alpha\gamma},
\end{align*}
for every $\tau_1,\tau_2\in [s,t]$, which shows that $g$ is $\alpha\gamma$-H\"older continuous in $[s,t]$. Since $\eta+\gamma\alpha>1$ we can apply Theorem \ref{thm:young_int} which implies that
\begin{align}
&\lim_{n\to+\infty}\sum_{j=1}^{m_n}\langle F_x(s^n_{j-1},y(s^n_{j-1})),\sigma(y(s^n_{j-1}))\rangle (x(s^n_j)-x(s^n_{j-1}))\notag\\
=&\int_s^t\langle F_x(u,y(u)),\sigma(y(u))\rangle dx(u),
\label{convergenza_4_chain_rule}
\end{align}
where the integral is well defined as Young integral. From \eqref{convergenza_1_chain_rule}-\eqref{convergenza_4_chain_rule} we conclude that
\begin{align*}
\lim_{n\to +\infty}I_{3,n}=\int_s^t\langle F_x(u,y(u)),Ay(u)\rangle du+ \int_s^t\langle F_x(u,y(u)),\sigma(y(u))\rangle dx(u).
\end{align*}

To complete the proof, we observe that $I_{4,n}$ converges to $0$ as $n$ tends to $+\infty$. This property can be checked arguing as we did for the term $I_{2,n}$, noting that
$F_x$ is uniformly continuous in $[0,T]\times y([0,T])$.

Summing up, we have proved that
\begin{align}
F(t,y(t))-F(s,y(s))
=&  \int_s^tF_t(u,y(u))du+\int_s^t\langle F_x(u,y(u)),Ay(u)\rangle du \notag\\
& +\int_s^t\langle F_x(u,y(u)),\sigma(y(u))\rangle dx(u),
\label{chain_rule_st}
\end{align}
for every $0<s<t\leq T$. As $s$ tends to $0^+$, the left-hand side converges to $F(t,y(t))-F(0,y(0))$.  As far as the right-hand side is concerned, the first and the third term converge to the corresponding integrals over $[0,t]$ since the functions $u\mapsto F_t(u,y(u))$ and $u\mapsto F_x(u,y(u))$ are continuous in $[0,T]$. As far as the second term in the right-hand side of \eqref{chain_rule_st} is concerned, thanks to \eqref{stima_lambda_mu} with $\mu=0$ we can apply the dominated convergence theorem which yields the convergence to the integral over $(0,t)$. The assertion in its full generality follows.
\end{proof}

The same arguments as in the proof of Theorem \ref{thm:chain_rula} and Corollary \ref{coro:int_sol_vett} give the following result.
\begin{corollary}
\label{coro:chain_rule_vett}
Let $\sigma_i:X\rightarrow X$ $(i=1,\ldots,m)$ satisfy Hypotheses \ref{ip_nonlineare}, let the paths $x_i\in C^\eta([0,T])$ $(i=1,\ldots,n)$ belong to $C^{\eta}([0,T])$, and let $y$ be the unique mild solution to \eqref{equazione_vettoriale} with $\psi\in X_\alpha$, with $\alpha+\eta>1$. Then, For any function $F$ satisfying the assumptions in Theorem $\ref{thm:chain_rula}$ it holds that
\begin{align}
\label{chain_rule_vett}
F(t,y(t))-F(s,y(s))
=&  \int_s^tF_t(u,y(u))du+\int_s^t\langle F_x(u,y(u)),Ay(u)\rangle du\notag\\
& +\sum_{i=1}^m\int_s^t\langle F_x(u,y(u)),\sigma_i(y(u))\rangle dx_i(u)
\end{align}
for every $(s,t)\in [0,T]$.
\end{corollary}

\section{Invariance of convex sets}
\label{sect-6}

Let $X$ be a Hilbert space with inner product $\langle\cdot,\cdot\rangle$, $A:D(A)\subset X\rightarrow X$ be a self-adjoint nonpositive closed operator which generates an analytic semigroup of bounded linear operators $(S(t))_{t\geq0}$ on $H$ and suppose that the results of the previous section hold true with   $X_\zeta=D((-A)^\zeta)$ for any $\zeta\geq0$.

We say that a closed convex set $K\subset X$ is invariant for the mild solution $y$ to \eqref{equazione_vettoriale}, if for any $\psi\in X_\alpha\cap K$, then $y(t)$ belongs to $K$ for any $t\in[0,T]$. Formula \eqref{chain_rule_vett} implies the following result.

\begin{proposition}
\label{prop-6.1}
Let Hypotheses $\ref{hyp-main}$, $\ref{ip:young_path}$, $\ref{ip_nonlineare}$ be fulfilled with $\eta+\alpha>1$. Let $\varphi\in X_\varepsilon$ for some $\varepsilon\in[0,1)$, $\psi\in X_\zeta$ for some $\zeta\in[\alpha,1)$ and let $K:=\{x\in X:\langle x,\varphi\rangle\leq 0\}$ be invariant for $y$. Then:
\begin{enumerate}[\rm (i)]
\item
If $\psi\in \partial K$ and $\eta\leq \zeta+\varepsilon\le 1$, then
\begin{align*}
\limsup_{t\rightarrow0^+}t^{-\beta}\sum_{i=1}^m\langle \varphi,\sigma_i(\psi)\rangle (x_i(t)-x_i(0))\leq 0,
\end{align*}
for $\beta\in[\eta,\zeta+\varepsilon)$ and
\begin{align*}
\sup_{\lambda>0}\limsup_{t\rightarrow0^+}t^{-\beta}\displaystyle\bigg  (
&-\int_0^t\frac{(\lambda+\langle \varphi,y(s)\rangle)_+}{\lambda}\langle (-A)^\varepsilon \varphi,(-A)^{1-\varepsilon}y(s)\rangle ds \\
&\;+\sum_{i=1}^m\langle \varphi,\sigma_i(\psi)\rangle (x_i(t)-x_i(0))\bigg )\leq 0,
\end{align*}
for $\beta\in[\zeta+\varepsilon,1]$.
\item
If $y(t_0)\in \partial K$ for some $t_0\in [0,T)$ and $\zeta+\varepsilon>1$, then
\begin{align*}
\limsup_{t\rightarrow t_0^+}|t-t_0|^{-\beta}\sum_{i=1}^m\langle \varphi,\sigma_i(y(t_0))\rangle (x_i(t)-x_i(t_0))\leq 0,
\end{align*}
if $\beta\in[\eta,1)$ and
\begin{align*}
\limsup_{t\rightarrow t_0^+}|t-t_0|^{-1}\sum_{i=1}^m\langle \varphi,\sigma_i(y(t_0))\rangle (x_i(t)-x_i(t_0))-\langle (-A)^{\varepsilon}\varphi, (-A)^{1-\varepsilon}y(t_0)\rangle\leq 0,
\end{align*}
if $\beta=1$. 
\end{enumerate}
\end{proposition}
\begin{remark}
{\rm If $t_0>0$ in $(ii)$ then $y(t)\in X_{1+\mu}$ for any $\mu\in[0,\eta+\alpha-1)$ and $t\in(0,T]$ (see Theorem \ref{es_mild_sol_gub}). Hence, the condition $\zeta+\varepsilon>1$ is automatically satisfied.}
\end{remark}

\begin{proof}[Proof of Proposition $\ref{prop-6.1}$]
For any $\lambda>0$ we introduce the function $F_{\lambda}:X\to X$, defined by $F_\lambda(x):=(\lambda+\langle \varphi,x\rangle)^2_+$, for any $x\in X$. As it is easily seen, each function $F_\lambda$ belongs to $C^{1,1}(X)$ and $DF_\lambda(x)=2(\lambda+\langle \varphi,x\rangle)_+\varphi$ for any $x\in X$. Further, for any $x\in K$ it holds that $F_\lambda(x)\leq \lambda^2$, and $F_\lambda(x)=\lambda^2$ if and only if $x\in \partial K$. For any $t_0\in[0,T)$, if $y(t_0)\in \partial K$ then from \eqref{chain_rule_vett} it follows that
\begin{align}
0
\geq & F_\lambda(y(t))-F_\lambda(y(t_0)) \notag \\
=&  2\int_{t_0}^t(\lambda+\langle \varphi,y(s)\rangle)_+\langle \varphi,Ay(s)\rangle ds
+2\sum_{i=1}^m\int_{t_0}^t(\lambda+\langle \varphi,y(s)\rangle)_+\langle \varphi,\sigma_i(y(s))\rangle dx_i(s) \notag \\
=& - 2\int_{t_0}^t(\lambda+\langle \varphi,y(s)\rangle)_+\langle (-A)^\varepsilon \varphi,(-A)^{1-\varepsilon}y(s)\rangle ds  \notag\\
& +2\sum_{i=1}^m\lambda\langle \varphi,\sigma_i(y(t_0))\rangle (x_i(t)-x_i(t_0))+{\mathscr R}_f(t_0,t),
\label{invarianza_1}
\end{align}
for any $t\in[t_0,T]$ and any $\lambda>0$, where in the last equality we have used \eqref{spezzamento_young_int} with $f(t)=(\lambda+\langle \varphi,y(t)\rangle)_+\langle \varphi,\sigma_i(y(t))\rangle$ for $t\in[0,T]$. We recall that from \eqref{stima_resto_Rf} it follows that ${\mathscr R}_f(t_0,t)=o(|t-t_0|)$ as $t\rightarrow t_0^+$, and from Remark \ref{rmk:remark_teorema_principale}$(ii)$ we have $y\in C((0,T];X_\mu)$ for any $\mu\in[0,\eta+\alpha)$.

Now, we separately consider the cases $(i)$ and $(ii)$.

\vspace{2mm}
{${\bf (i)}$}. Fix $\psi\in \partial K\cap X_\zeta$ with $\eta\leq \zeta+\varepsilon\leq 1$. From \eqref{invarianza_1}, with $t_0=0$ we get
\begin{align}
0
\geq &- 2\int_0^t(\lambda+\langle \varphi,y(s)\rangle)_+\langle (-A)^\varepsilon \varphi,(-A)^{1-\varepsilon}y(s)\rangle ds  \notag\\
& +2\sum_{i=1}^m\lambda\langle \varphi,\sigma_i(\psi)\rangle (x_i(t)-x_i(0))+{\mathscr R}_f(0,t),
\label{invarianza_2}
\end{align}
for any $t\in[0,T]$. From \eqref{stima_sol_Xr} with $r=1-\varepsilon$ we infer that there exists a positive constant $c$, independent of $s$, such that
\begin{align}
\label{stima_inv_2}
\left|(\lambda+\langle \varphi,y(s)\rangle)_+\langle (-A)^\varepsilon \varphi,(-A)^{1-\varepsilon}y(s)\rangle\right|\leq \lambda cs^{\zeta+\varepsilon-1}, \quad s\in(0,T].
\end{align}
Let $\beta\in[\eta,1]$. By dividing both the sides of \eqref{invarianza_2} by $t^\beta$ and by $\lambda$, and taking \eqref{stima_inv_2} into account, the assertion follows easily.

{${\bf (ii)}$}. Fix $t_0\in [0,T)$ with $y(t_0)\in \partial K$ and $\zeta+\varepsilon>1$.
Since $1-\varepsilon<\zeta$ it follows that $y$ is continuous up to $0$ with values in $X_{1-\varepsilon}$. Indeed, from \eqref{stime_smgr}$(b)$ we get $|S(t)\psi-\psi|_{1-\varepsilon}\leq Ct^{\zeta+\varepsilon-1}$ for every $t\in [0,T]$, and from \eqref{reg_sol-aaa} and the smoothness of $\mathscr I_{S\sigma(y)}$ at $t=0$ we infer that $y\in C_b([0,T];X_{1-\varepsilon})$. As a consequence, the function
$s\mapsto (\lambda+\langle \varphi,y(s)\rangle)_+\langle (-A)^\varepsilon \varphi,(-A)^{1-\varepsilon}y(s)\rangle$ belongs to $C_b([0,T])$.
Let $\beta\in[\eta,1]$. Dividing \eqref{invarianza_1} by $|t-t_0|^\beta$ and $\lambda$, and letting $t\to t_0^+$, the assertion follows also in this case.
\end{proof}

\begin{remark} {\rm In \cite{CouMar16}, the problem of the invariance of a general convex set $K$ is treated for a ordinary differential equation driven by a rough path. In such a paper, the state space is finite dimensional and no unbounded operators are involved in the  equation but  both the noise, which is $\theta$-H\"older with $\theta>1/3$, and the set $K$ are much more general than here. In the present paper, we just wanted to show how the  availability of both a classical solution and a chain rule can be exploited to tackle the problem of the invariance of convex sets, when an unbounded operator $A$ is involved.
As a matter of fact, the general problem of the invariance of a convex set $K$ with respect to a general infinite dimensional evolution equation driven by a rough trajectory is still unexplored (see \cite{CaDa} and the references therein for corresponding results in the case of classical evolution equations). It is also worth mentioning that, in \cite{CouMar16}, necessary conditions for the invariance are proved under an assumption specifying the genuine rough behaviour of the path (i.e., requiring that $|x(t)-x(0)| \sim t^{\theta} $  as $t$ tends to $0$).}
\end{remark}

\section{An example}
\label{sect-examples}

In this section, we provide an example to which our results apply.

Let $A$ be the realization in $X=C_b(\R^d)$ of the second-order elliptic operator
\begin{eqnarray*}
{\mathcal A}=\sum_{i,j=1}^dq_{ij}D_{ij}+\sum_{j=1}^db_jD_j+c,
\end{eqnarray*}
with coefficients bounded and $\beta$-H\"older continuous on $\R^d$, for some $\beta\in (0,1/2)$, and $\sum_{i,j=1}^dq_{ij}(x)\xi_i\xi_j\ge\mu |\xi|^2$
for every $x,\xi\in\R^d$ and some positive constant $\mu$. Note that $D(A)=\Big\{u\in C_b(\R^d)\cap\bigcap_{p<+\infty}W^{2,p}_{\rm loc}(\R^d): {\mathcal A} u\in C_b(\R^d)\Big\}$.
For every $\alpha\in (0,2)\setminus\{1/2,1\}$, we take $X_{\alpha}=C^{2\alpha}_b(\R^d)$ endowed with the classical norm of $C^{2\alpha}_b(\R^d)$.
Moreover, we take as $X_{1/2}$ the Zygmund space of all bounded functions $g:\R^d\to\R$ such that $[g]=\displaystyle\sup_{x\neq y}\frac{|g(x)+g(y)-2g(2^{-1}(x+y))|}{|x-y|}<+\infty$ endowed with the norm $\|g\|_{X_{1/2}}=\|g\|_{\infty}+[g]$.
It is well known that $A$ generates an analytic semigroup on $C_b(\R^d)$ and $X_{\alpha}$ is the interpolation space of order $\alpha$ between $X$ and $X_1=D(A)$. We refer the reader to e.g., \cite[Chapters 3 and 14]{lor-rha}.
Finally, we fix a function $\hat\sigma\in C^2_b(\R)$ and note that the function $\sigma:X\to X$, defined by $\sigma(f)=\hat\sigma\circ f$ satisfies Hypothesis \ref{ip_nonlineare} is satisfied for every $\alpha\in (0,1/2)$, with
$L_{\sigma}^{\alpha}=\|\sigma\|_{{\rm Lip}(\R)}$. Since the assumptions of Theorem \ref{es_mild_sol_gub} are satisfied, we conclude that, for every $\psi\in C^{\alpha}_b(\R^d)$ ($\alpha\in (0,1)$), there exists a unique classical solution $y$ to Problem \eqref{omo_equation}.


\begin{thebibliography}{99}
\bibitem{CaDa} P. Cannarsa, G. Da Prato, {\em Stochastic viability for regular closed sets in Hilbert spaces.} Rend. Lincei Mat. Appl., \textbf{22} (2011), 337--346.

\bibitem{CouMar16} L. Coutin, N. Marie, {\em Invariance for rough differential equations}, Stochastic Processes and their Applications \textbf{127} (2017) 2373--2395.

\bibitem{DGT12}
A. Deya, M. Gubinelli, S. Tindel,
{\em Non-linear rough heat equations}, Probab. Theory Related Fields \textbf{153} (2012), 97--147.

\bibitem{FrHa} P. Fritz, M Hairer, A Course on Rough Paths
With an Introduction to Regularity Structures, Springer Universitext 2014.

\bibitem{G04}
M. Gubinelli,
{\em Controlling rough paths}, J. Funct. Anal., \textbf{216} (2004) 86--140.

\bibitem{gubinelli-panorama} M. Gubinelli,
{\em A panorama of Singular SPDEs}, Proceedings of the International Congress of Mathematicians (ICM 2018), World Scientific, 2019.

\bibitem{GLT06}
M. Gubinelli, A. Lejay, S. Tindel,
{\em Young integrals and SPDEs}, Potential Anal. \textbf{25} (2006), 307--326.

\bibitem{GT10}
M. Gubinelli, S. Tindel,
{\em Rough evolution equations}, Ann. Probab. \textbf{38} (2010), 1--75.

\bibitem{Lej03}
A. Lejay, {\em An introduction to rough paths.} In: Seminaire
de Probabilit\'es  XXXVII. Lecture Notes in Mathematics  \textbf{1832}, 1--59. Springer, Berlin Heidelberg New York (2003)

\bibitem{lor-rha}
L. Lorenzi, A. Rhandi,
Semigroups of bounded operators and second-order elliptic and parabolic partial differential equations, CRC Press, 2021.

\bibitem{LU95}
A. Lunardi, Analytic semigroups and optimal regularity in parabolic problems,
Modern Birkh\"{a}user Classics, Birkh\"{a}user/Springer Basel AG, Basel, 1995.

\bibitem{Lyo98}
T. J. Lyons,
{\em Differential equations driven by rough signals},
Rev. Mat. Iberoamericana \textbf{14} (1998), no. 2, 215--310.

\bibitem{MaNu03}
B. Maslowski, D. Nualart, {\em Evolution equations driven by a fractional Brownian motion.} J. Funct. Anal. \textbf{202} (2003), 277--305 .

\bibitem{You36}
L.C. Young,
{\em An inequality of the Hlder type, connected with Stieltjes integration},
Acta Math. \textbf{67} (1936), no. 1, 251--282.

\bibitem{Zah98} M.
Z\"{a}hle, {\em Integration with respect to fractal functions and stochastic calculus. I.} Probab. Theory Related Fields \textbf{111} (1998), 333--374.
\end{thebibliography}
\end{document}